\documentclass{amsart}
\usepackage[T1]{fontenc}
\usepackage{geometry}
\usepackage{mathrsfs}
\usepackage{verbatim}
\usepackage{amsmath,amsfonts,amsthm,amssymb,epsfig}
\usepackage[all]{xy}
\newcommand{\proet}{\mathrm{pro\acute{e}t}}
\usepackage{enumerate}
\usepackage[usenames,dvipsnames]{xcolor}
\usepackage[bookmarks=true,colorlinks=true, linkcolor=Mulberry, citecolor=Periwinkle]{hyperref}
\usepackage{cleveref} 
\usepackage{relsize} 
\usepackage[bbgreekl]{mathbbol} 
\usepackage{amsfonts}  
\DeclareSymbolFontAlphabet{\mathbb}{AMSb} %to ensure that the meaning of \mathbb does not change
\DeclareSymbolFontAlphabet{\mathbbl}{bbold} 
\newcommand{\Prism}{{\mathlarger{\mathbbl{\Delta}}}}
\newcommand{\Prismh}{\widehat{\Prism}}

\newcommand{\triplearrows}{\begin{smallmatrix} \to \\ \to \\ 
\to \end{smallmatrix} }

\renewcommand{\mod}{\mathrm{Mod}}
\newcommand{\fil}{\mathrm{Fil}}

\newcommand{\sF}{\mathcal{F}}

\renewcommand{\sp}{\mathrm{Sp}}
\usepackage{graphicx}
\usepackage{float}
\newcommand{\gr}{\mathrm{gr}}
\newcommand{\et}{\mathrm{\acute{e}t}}

\newtheorem{theorem}{Theorem}[section]
\newcommand{\psh}{\mathrm{PSh}}

\newtheorem{lemma}[theorem]{Lemma}
\newtheorem{proposition}[theorem]{Proposition}
\newtheorem{corollary}[theorem]{Corollary}
\renewcommand{\hom}{\mathrm{Hom}}
\theoremstyle{definition}

\newtheorem{question}[theorem]{Question}
\newcommand{\fun}{\mathrm{Fun}}
\newtheorem{construction}[theorem]{Construction}
\newcommand{\spec}{\mathrm{Spec}}
\newtheorem{definition}[theorem]{Definition}
\newtheorem{example}[theorem]{Example}
\newtheorem{remark}[theorem]{Remark}

\newcommand{\TC}{\mathrm{TC}}
\newcommand{\TP}{\mathrm{TP}}

\newcommand\bcdot{\ensuremath{%
  \mathchoice%
   {\mskip\thinmuskip\lower0.2ex\hbox{\scalebox{1.5}{$\cdot$}}\mskip\thinmuskip}}%
   {\mskip\thinmuskip\lower0.2ex\hbox{\scalebox{1.5}{$\cdot$}}\mskip\thinmuskip}%   
   {\lower0.3ex\hbox{\scalebox{1.2}{$\cdot$}}}%
   {\lower0.3ex\hbox{\scalebox{1.2}{$\cdot$}}}%
   }

\begin{document}

\newcommand{\ksel}{K^{Sel}}
\newcommand{\qsyn}{\mathrm{QSyn}_{\mathbb{Z}}}
\newcommand{\cycsp}{\mathrm{CycSp}}
\newcommand{\cat}{\mathrm{Cat}}
\newcommand{\mot}{\mathrm{mot}}
\newcommand{\TR}{\mathrm{TR}}
\newcommand{\THH}{\mathrm{THH}}
\title{On $K(1)$-local $\TR$}
\author{Akhil Mathew}
\date{\today}
\email{amathew@math.uchicago.edu}
\address{Department of Mathematics, University of Chicago, 5734 S University
Ave, Chicago, IL 60605 USA}

\begin{abstract}
We discuss some general properties of $\TR$ and its $K(1)$-localization. We
prove that after $K(1)$-localization, $\TR$ of $H\mathbb{Z}$-algebras is a
truncating invariant in the sense of Land--Tamme, and deduce $h$-descent
results. We show that for regular rings in mixed characteristic, $\TR$ is
asymptotically $K(1)$-local, extending results of Hesselholt--Madsen.  As an application of these methods and recent advances in the theory
of cyclotomic spectra, 
we construct an analog of Thomason's spectral sequence relating $K(1)$-local
$K$-theory  and \'etale cohomology for $K(1)$-local $\TR$. 

\end{abstract}

\maketitle
\section{Introduction}
The topological cyclic homology, $\TC(R)$, of a ring (or ring spectrum) $R$
is a basic invariant 
introduced by B\"okstedt--Hsiang--Madsen \cite{BHM} 
(see also Dundas--Goodwillie--McCarthy \cite{DGM} and Nikolaus--Scholze
\cite{NS18})  with many
applications in algebraic $K$-theory. Its $p$-adic completion $\TC(R;
\mathbb{Z}_p)$ arises as the fixed points
of an operator called Frobenius on another invariant $\TR(R; \mathbb{Z}_p)$, which plays a
central role in the approach to $\TC$ via equivariant stable homotopy theory. 
The construction $\TR(-; \mathbb{Z}_p)$ is often of arithmetic significance; for instance, the
foundational calculations \cite{HM03, HM04, GH06} of the $p$-adic $K$-theory of local
fields $F$ are based on the relationship between $\TR(\mathcal{O}_F;
\mathbb{Z}_p)$ and the de
Rham--Witt complex with log poles of $\mathcal{O}_F$. 

In this paper, we prove some structural results about $\TR $ and how it
relates to its $K(1)$-localization, throughout at an implicit prime number $p$. 
We will only consider $p$-typical $\TR$ and after $p$-adic completion. 
The operation of $K(1)$-localization 
when applied to algebraic $K$-theory is dramatically simplifying, as shown by
Thomason \cite{Tho85, TT90}; in particular, $K(1)$-local $K$-theory satisfies
\'etale descent and admits a descent spectral sequence from \'etale cohomology under mild hypotheses, cf.~\cite{CM19}
for a modern account. 
Here we study analogs of some of these properties for $L_{K(1)} \TR(-)$. 
Our starting point is the following. 

\begin{theorem} 
\label{truncatingthmintro}
As a functor on connective $H\mathbb{Z}$-algebras, 
$L_{K(1)} \TR(-)$
is a
truncating invariant in the sense of Land--Tamme \cite{LT19}. In other words, if
$R$ is a connective $H\mathbb{Z}$-algebra, then $L_{K(1) } \TR(R)
\xrightarrow{\sim} L_{K(1)} \TR(\pi_0 R)$. \end{theorem}

\Cref{truncatingthmintro} refines results of \cite{LMT, BCM}. In
\emph{loc.~cit.,}~it
is shown 
that $L_{K(1)} K(-)$ and (equivalently) $L_{K(1)} \TC(-)$ are truncating
invariants of connective
$H\mathbb{Z}$-algebras, which ultimately follows from the claim
\begin{equation} L_{K(1)}
K(\mathbb{Z}/p^n) =0, \quad n \geq 1, \label{k1vanish} \end{equation} 
and more generally (and consequently) for any $H\mathbb{Z}$-algebra $R$,
\begin{equation} 
L_{K(1)} K(R) \xrightarrow{\sim} L_{K(1)} K(R[1/p]). \label{k1invertp}
\end{equation} 
In \cite{BCM}, \eqref{k1vanish} is proved via a calculation in prismatic cohomology; in
\cite{LMT}, \eqref{k1vanish} is proved using some unstable chromatic homotopy theory. 
Our proof of \Cref{truncatingthmintro} (which also gives a new proof of
\eqref{k1vanish}) is based on a direct $\TC$-theoretic argument via estimation
of exponents of nilpotence of the Bott element; in fact, it yields a slightly
stronger result (\Cref{truncatingthm} below).

The property of being truncating yields many pleasant features of the
construction $L_{K(1)} \TR(-)$: by \cite{LT19}, one obtains $cdh$-descent and
excision. Since we are working $K(1)$-locally, we can combine this with results
of \cite{CMNN} to obtain $h$-descent: 

\begin{theorem} 
\label{hdescintro}
Any $K(1)$-local localizing invariant which is truncating, such as $L_{K(1)}
\TR(-)$, satisfies $h$-descent
on qcqs schemes. 
\end{theorem}

In particular, 
$L_{K(1)}\TR(-)$ satisfies \'etale descent. This is not so surprising, 
since $\TR(-; \mathbb{Z}_p)$ itself (like all Hochschild-theoretic invariants) actually satisfies
flat descent, cf.~\cite[Sec.~3]{BMS2}. However, 
\Cref{hdescintro}
(together with \eqref{k1vanish}) leads to \'etale descent in the generic fiber. 
Since $\TR(R; \mathbb{Z}_p)$ of a ring $R$  depends only on the (derived)
$p$-adic completion of $R$, we can informally view $L_{K(1)} \TR(R)$ as an invariant of the ``rigid space''
associated to $R[1/p]$.

\begin{example}[Galois descent in the generic fiber] 
\label{galoisgenericex}
Let $R \to S$ be a finite, finitely presented map of rings. 
Suppose we have a finite group $G$ acting on $S$ such that $R[1/p] \to S[1/p]$
is $G$-Galois. Then 
for any $K(1)$-local localizing invariant $E$ which is truncating, we have
\( E(R) \xrightarrow{\sim} E(S)^{hG}.  \)
\end{example}

Recall that the Lichtenbaum--Quillen conjecture, refined by the
Beilinson--Lichtenbaum conjecture proved by Voevodsky--Rost (cf.~\cite{HW19}
for an account), predicts that for 
$\mathbb{Z}[1/p]$-algebras $A$ satisfying 
mild finiteness conditions, the $p$-adic $K$-theory $K(A; \mathbb{Z}_p)$ is
``asymptotically $K(1)$-local:'' that is, the map $K(A; \mathbb{Z}_p) \to
L_{K(1)} K(A; \mathbb{Z}_p)$ is an equivalence in high enough degrees. 
We next discuss analogs of such statements for $\TR(R; \mathbb{Z}_p)$ for $p$-adic
rings $R$. 
Indeed, 
in \cite{HM03, HM04}, it is shown that
if $R$ is smooth of relative dimension $d$ over a discrete valuation ring $\mathcal{O}_K$ of mixed
characteristic  with perfect
residue field of characteristic $p > 2$, then 
$\TR(R; \mathbb{F}_p) \to L_{K(1)} \TR(R; \mathbb{F}_p)$ is $d$-truncated; more
precisely, this is a consequence of the relationship shown in \emph{loc.~cit.}~with the absolute de Rham--Witt complex. 
We prove this asymptotic $K(1)$-locality more generally for regular rings satisfying $F$-finiteness
hypotheses
from the Beilinson--Lichtenbaum conjecture  applied to the generic fiber as well as the connection between $\TR$ and $p$-typical curves
\cite{Hes96}. We expect that there should be a purely $p$-adic proof of
this result (as in \cite{HM03, HM04} in the smooth case).  

\begin{theorem} 
Let $R$ be a $p$-torsionfree excellent regular noetherian ring. Suppose that $R/p$ is finitely generated
as a module over its subring of $p$th powers. 
Suppose furthermore that
for all $\mathfrak{p} \in \spec(R)$ containing $(p)$, we have
$\dim R_{\mathfrak{p}} + \log_p [ \kappa(\mathfrak{p}):\kappa(\mathfrak{p})^p]
\leq d$ for some $d \geq 0$. 
Then the map $\TR(R; \mathbb{F}_p) \to L_{K(1)} \TR(R; \mathbb{F}_p)$ is 
$(d-1)$-truncated. 
\label{TRtruncapprox}
\end{theorem} 

Using the theory of topological Cartier modules of Antieau--Nikolaus \cite{AN}, we relate 
the property of $\TR(R; \mathbb{F}_p)$ being ``asymptotically $K(1)$-local'' to
the extensively studied Segal
conjecture for $\THH(R)$, i.e., the condition that the cyclotomic
Frobenius $\varphi \colon \THH(R; \mathbb{F}_p) \to \THH(R;
\mathbb{F}_p)^{tC_p}$
should be an equivalence in high degrees. 
In particular, we obtain a version of the Segal conjecture for $\THH(R)$ when
$R$ is regular. 
We expect that there should be a filtered version of this
statement, using the motivic filtrations of Bhatt--Morrow--Scholze \cite{BMS2}.

\begin{corollary} 
Let $R$ be as in 
\Cref{TRtruncapprox}. Then the cyclotomic Frobenius $\varphi \colon\THH(R; \mathbb{F}_p) \to
\THH(R; \mathbb{F}_p)^{tC_p}$ is $(d-1)$-truncated. 
\end{corollary}

\newcommand{\arc}{\mathrm{arc}_{{p}}}

Finally, we study the analog of Thomason's spectral sequence \cite{Tho85, TT90}
from \'etale cohomology to $K(1)$-local algebraic $K$-theory. 
For a scheme $X$ over $\mathbb{Z}[1/p]$ satisfying mild finiteness conditions
(to wit: $X$ should be qcqs of finite Krull dimension, with a uniform bound on
the mod $p$ virtual cohomological dimensions of the residue fields,
\cite{RO06, CM19}), one has a convergent spectral sequence
\[ E_2^{s,t} =H^s(X, \mathbb{Z}_p(t)) \implies \pi_{2t-s} L_{K(1)} K(X).  \]

We can construct a similar spectral sequence for $L_{K(1)} \TR$ under
significantly stronger finiteness and regularity
conditions, arising from a natural filtration. 
To formulate the $E_2$-term (or the graded pieces of this filtration), we use the $\arc$-topology of \cite{BM18}. 

\begin{definition}[The $\arc$-topology and $\arc$-cohomology] 
\label{arcpdefinitionintro}
We say that a map of derived $p$-complete rings $R \to R'$ is an \emph{$\arc$-cover} if 
every map $R \to V$ for $V$ a rank $1$ valuation ring  which is $p$-complete and such that $p \neq 0$,
there exists an extension of rank $1$ valuation rings $V \to W$ and a
commutative diagram
\[ \xymatrix{
R \ar[d]  \ar[r] &  R' \ar[d]  \\
V \ar[r] &  W 
}.\]
The $\arc$-topology is the finitary 
Grothendieck topology on the opposite of the category of derived
$p$-complete rings defined such that a family $\left\{R \to R'_\alpha\right\}_{\alpha \in A}$ is a
covering family if and only if there exists a finite subset $A' \subset A$ such
that $R \to \prod_{\alpha \in A'} R'_\alpha$ is an $\arc$-cover. 

Given any functor $F$
from derived $p$-complete rings to abelian groups, 
we let $R \Gamma_{\arc}( \spec(R), F(-))$ denote the $\arc$-cohomology of $F(-)$
on a derived $p$-complete ring $R$.\footnote{For set-theoretic reasons, to define $\arc$-cohomology
we should fix a cutoff cardinal. We will only consider situations where the
choice of cutoff cardinal does not affect the result.} 
\end{definition} 

\begin{example} 
\begin{enumerate}
\item  
We can consider the $\arc$-cohomology of the structure presheaf $\mathcal{O}$,
$R \Gamma_{\arc}( \spec(R), \mathcal{O})$. This is closely related to the
perfectoidizations considered in \cite[Sec.~7--8]{Prisms}, which work with the
$p$-complete
$\mathrm{arc}$-topology rather than the $\arc$-topology. 
For $R = \mathbb{Z}_p$, it is not difficult to see that this is the continuous
$\mathrm{Gal}(\mathbb{Q}_p)$-homotopy invariants of the derived
saturation $(\mathcal{O}_C)_*$ for $C = \widehat{\overline{\mathbb{Q}_p}}$ (in
the sense of almost ring theory \cite{GR03}). 
\item
We consider the $\arc$-cohomology of the Witt vector presheaf $W(\mathcal{O})$, denoted 
$R \mapsto R \Gamma_{\arc}( \spec(R), W(\mathcal{O}))$, as well as its $p$-adic Tate twists
$W(\mathcal{O})(i)$ for $i \in \mathbb{Z}$. 
\end{enumerate}
\end{example}

\begin{theorem} 
\label{sseqthmintro}
Let $K$ be a complete nonarchimedean field of  mixed characteristic $(0, p)$
with  ring of integers $\mathcal{O}_K$ and residue field $k$ with $[k: k^p] <
\infty$. 
Suppose either $K$ is discretely valued or $K$ is perfectoid. 
Let 
$R$ be a formally smooth $\mathcal{O}_K$-algebra. 
Then there exists a natural convergent, exhaustive
$\mathbb{Z}$-indexed descending filtration 
$\mathrm{Fil}^{\geq \ast} L_{K(1)} \TR(R) $ on $L_{K(1)}
\TR(R)$ such that 
\begin{equation}  \label{motivicfilK1TR} \gr^{i} L_{K(1)} \TR(R) \simeq R \Gamma_{\arc}( \spec(R),
W(\mathcal{O})(i))[2i]. \end{equation}
\end{theorem}

\Cref{sseqthmintro} is effectively an \'etale hyperdescent (in the generic
fiber) result together with the calculation for perfectoids. 
For illustration, we specialize to the case where 
$R  = \mathcal{O}_K$, for $K$ discretely valued with perfect residue field. 
Then the filtration \eqref{motivicfilK1TR} arises via a type of pro-Galois descent in the
generic fiber. 
If $L/K$ is $G$-Galois, \Cref{galoisgenericex} and \Cref{TRtruncapprox} imply that 
$\TR(\mathcal{O}_K; \mathbb{F}_p) \to \TR(\mathcal{O}_L; \mathbb{F}_p)^{hG}$ is a $0$-truncated map. 
However, this does not help with passage to $\overline{K}$ since $\TR$ does not
commute with filtered colimits; note that $\TR$ 
has a simple form for $\mathcal{O}_{\overline{K}}$, cf.~\cite{Hes06}. 
Using again the theory of topological Cartier modules as a ``decompletion'' of
the theory of cyclotomic spectra, cf.~\cite{AN},
we prove the following pro-Galois result (cf.~\Cref{DVRex}): 

\begin{theorem} 
Let $K$ be a complete, discretely valued field of characteristic zero with
perfect residue field $k$ of odd
characteristic $p$. 
Let $\TR(\mathcal{O}_K|K)$ denote the cofiber of the transfer map 
$\TR(k) \to \TR(\mathcal{O}_K)$. 
Then the natural map induces an equivalence
\[ \TR(\mathcal{O}_K|K; \mathbb{F}_p) \to \tau_{\geq 0}\mathrm{Tot} \left(  \TR(
\mathcal{O}_{\overline{K}}; \mathbb{F}_p)
\rightrightarrows 
\TR( \mathcal{O}_{\overline{K} \otimes_K \overline{K}}; \mathbb{F}_p)
\triplearrows \dots  \right). 
\]
\end{theorem}

The idea that $\TR$ should satisfy this type of pro-Galois descent in the
generic fiber
is expressed in \cite{HesICM}; in particular Conjecture 5.1 of \emph{loc.~cit.}~predicts a related (but stronger) conclusion at the level of homotopy groups (in particular, the
vanishing of higher Galois cohomology groups in the associated descent spectral
sequence with $\mathbb{Q}_p/\mathbb{Z}_p$ coefficients). 

\subsection*{Conventions}

We write $\sp$ for the $\infty$-category of spectra and $\mathbb{S}$ for the
sphere spectrum. 
We use the theory of cyclotomic spectra in the form developed in
\cite{NS18}, as well as the theory of topological Cartier modules developed in
\cite{AN}; we write $\cycsp$ for the $\infty$-category of cyclotomic spectra. 
We write $\TR$ for $p$-typical $\TR$. 
Given an $E_\infty$-ring $B$, a homogeneous element $x \in \pi_*(B)$, and a
$B$-module $M$, we
often write $M/x$ for the cofiber of multiplication by $x$ on $M$. 
In the case $x =p$, we will often write this by $; \mathbb{F}_p$, e.g., $\THH(B;
\mathbb{F}_p)$ refers to the cofiber of $p$ on $\THH(B)$. 
A $B$-algebra always refers, unless otherwise specified, to an $E_1$-algebra in
$B$-modules.

\subsection*{Acknowledgments}

I would like to thank 
Benjamin Antieau, Lars Hesselholt, Matthew Morrow, Thomas Nikolaus, 
Wies{\l}awa Nizio\l, and Peter Scholze for many
helpful conversations related to the subject of this paper. 
I would especially like to thank Bhargav Bhatt and Dustin Clausen for their
collaboration in \cite{BCM} and many helpful discussions. 
Finally, I would like to thank Benjamin Antieau, Lennart Meier, Georg Tamme,
and the referee for several comments and corrections on an earlier version. 
{This work was done while the author was a Clay Research Fellow.}

\section{Generalities on $K(1)$-local truncating invariants}
\label{generalK1loctrunc}

Let $B$ be a base connective $E_\infty$-ring. 
In this section, we work with localizing invariants on small $B$-linear
idempotent-complete stable $\infty$-categories. Unlike in \cite{BGT13}, we do
not assume compatibility with filtered
colimits, so for us a localizing
invariant is simply a functor from (small, idempotent-complete) $B$-linear
stable $\infty$-categories to spectra which carries Verdier quotient sequences
to cofiber sequences. 
Following \cite{LT19}, we say that  such a localizing invariant $E$ is
\emph{truncating} if for every connective $B$-algebra $A$, we have 
$E(A) \xrightarrow{\sim} E(H \pi_0 A)$. 
This implies \cite[Th.~B]{LT19} that $E$ satisfies excision. 

\begin{example} 
The constructions $L_{K(1)} K(-), L_{K(1)} \TC(-)$ are truncating on connective
$H\mathbb{Z}$-algebras, as verified in \cite{LMT}. 
For commutative $p$-complete rings, the two invariants actually agree  (we do not know if
this is true for noncommutative $p$-complete rings, cf.~\cite[Question 2.20]{BCM}). 
Below we will show that $L_{K(1)}
\TR(-)$ is truncating. 
\end{example}

In the rest of the section, we will assume for simplicity of notation that $B$
is discrete; by the assumption of truncatedness, this does not affect any of the
results. 

\begin{proposition} 
\label{isogenyinv}
Let $E$ be a $K(1)$-local localizing invariant of $B$-linear
$\infty$-categories which is truncating. 
Then, on the category of discrete $B$-algebras: 
\begin{enumerate}
\item $E$ is nilinvariant. 
\item $E$ annihilates any $B$-algebra $C$   which is annihilated by a power
of $p$. 
\item Let $A \to A'$ be a map of $B$-algebras which is a $p$-isogeny. Then $E(A) \to E(A')$ is an equivalence. 
\end{enumerate}
\end{proposition} 
\begin{proof} 
For (1), the fact that $E$ is nilinvariant follows from \cite[Th.~B]{LT19}.  
For (2), 
since $E$ is nilinvariant, we may assume $C$ is an $\mathbb{F}_p$-algebra, so
that $E(C)$ is by the theory of non-commutative motives \cite{BGT13} a $K(\mathbb{F}_p; \mathbb{Z}_p)  = H\mathbb{Z}_p$-module (the
last identification by \cite{Qui72K}); since
$E$ is $K(1)$-local we get $E(C) = 0$. 

For (3),
the kernel of $A \to A'$ is annihilated by a power of $p$, so 
by (2)
(and excision) we can assume that $A \subset A'$. 
Let $n \gg 0$, so $p^n A' \subset A$. Then the diagram
\[ \xymatrix{
A \ar[d]  \ar[r] &  A' \ar[d]  \\
A/(p^n A' \cap A) \ar[r] &  A'/p^n A'
}\]
is a Milnor square of rings.  Applying the localizing invariant $E$   and using
excision and (2) again, we conclude (3).   
\end{proof} 

In the next result, we use Voevodsky's $h$-topology for
possibly non-noetherian schemes; in other words, the topology generated by finitely
presented  $v$-covers 
(also called universally subtrusive morphisms),
cf.~\cite{Rydh} or \cite[Sec.~2]{BSproj}. 

\begin{theorem}[$h$-descent for truncating $K(1)$-local invariants] 
Let $E$ be a $K(1)$-local localizing invariant on $B$-linear $\infty$-categories
which is truncating. 
Then $E$ satisfies $h$-descent on quasi-compact and quasi-separated (qcqs) $B$-schemes. 
\label{hdescthm}
\end{theorem} 
\begin{proof} 
By the results of \cite[App.~A]{LT19}, $E$ satisfies $cdh$-descent, and in
particular satisfies excision for abstract blow-up squares. 
The results of \cite{CMNN} imply that $E$ satisfies finite locally free descent. 
Since $E$ also satisfies Nisnevich descent (as does any localizing
invariant \cite{TT90}), we
obtain that $E$ satisfies fppf descent thanks to \cite[Tag
05WN]{stacks-project}. 
By \cite[Th.~2.9]{BSproj}, $h$-descent (for any sheaf) is implied by fppf descent and excision
for abstract blow-up squares. 
Combining these facts, we conclude. \end{proof} 

\begin{example} 
\label{rigidgenericfiberproper}
Let $E$ be as above. 
Let $\pi\colon  X'  \to  X$ be a finitely presented proper morphism (e.g., a
finitely presented closed immersion) of qcqs $B$-schemes such that
$X'[1/p] \xrightarrow{\sim} X[1/p]$. Then 
$E(X)  \xrightarrow{\sim}  E(X')$. 
In fact, 
by cdh descent, we have a pullback square
\[ \xymatrix{
E(X) \ar[d] \ar[r] &  E(X') \ar[d]  \\
E(X \otimes \mathbb{F}_p) \ar[r] &  E(X' \otimes \mathbb{F}_p)
},\]
and the terms on the bottom vanish 
by 
\Cref{isogenyinv}. 
\end{example}

In the next result, we use the notion of nilpotence of a group action,
cf.~\cite[Sec.~4.1]{Examples} or \cite[Def.~2.17]{CM19} for accounts, or
\cite{MNN17} for the general setup in equivariant stable homotopy theory. 
Let $G$ be a finite group. 
The collection of \emph{nilpotent} objects of the $\infty$-category $\fun(BG,
\sp)$ is the thick subcategory generated by the objects which are induced from
the trivial subgroup. For an algebra object of $\fun(BG, \sp)$, nilpotence
holds if and only if the Tate construction vanishes. 
A module over a nilpotent algebra object in $\fun(BG, \sp)$ is itself nilpotent.

\begin{corollary}[Galois descent in the generic fiber] 
Let $E$ be a $K(1)$-local localizing invariant of $B$-linear
$\infty$-categories which is truncating. 
Let $R \to S$ be a finite and finitely presented map of (commutative, discrete) $B$-algebras. Let $G$ be a finite group acting on $S$ via
$R$-algebra maps. Suppose that $R[1/p] \to S[1/p]$ is $G$-Galois. 
Then the natural map induces an equivalence
$E(R)  \xrightarrow{\sim} E(S)^{hG}$. 
Moreover, the $G$-action on $E(S)$ is nilpotent. \label{galoisdescgeneric}
\end{corollary} 
\begin{proof} 
Replacing $S$ with $S \times R/p$, we may assume without loss of generality that
$R \to S $ is an $h$-cover. 
Then, by \Cref{hdescthm}, we have
\begin{equation} \label{CMNNdesc} E(R) \simeq \mathrm{Tot}(  E(C(R \to
S)^\bullet)) = \mathrm{Tot}(E(S) \rightrightarrows E(S \otimes_R S ) \triplearrows
\dots ). \end{equation}
Since the group $G$ acts on $S$, we have a natural map of cosimplicial rings
from the \v{C}ech nerve $C(R \to S)^\bullet$ to the standard resolution 
$S \rightrightarrows \prod_G S \triplearrows $
for $G$ acting
on $S$ (which calculates $S^{hG}$). 
This map of cosimplicial rings is an isogeny in each degree: for example, in degree $1$, $S \otimes_R S \to \prod_G
S$ is an isogeny because it is a map of finitely presented $R$-modules 
(cf.~\cite[Tag 0564]{stacks-project})
which
induces an isomorphism after inverting $p$ thanks to the Galois hypothesis. 
Therefore, by \Cref{isogenyinv},
we find that 
the map induces an equivalence after applying $E$, and we find 
from \eqref{CMNNdesc},
\[ E(R) \xrightarrow{\sim} \mathrm{Tot}( E(S) \rightrightarrows E( \prod_G S) \triplearrows
\dots ) = E(S)^{hG},\]
which is the desired claim. 
Finally, to see that the 
$G$-action on $E(S)$ 
is nilpotent, we use that $E(S)$ is a module $G$-equivariantly over 
$L_{K(1)} K(S) \xrightarrow{\sim} L_{K(1)} K(S[1/p])$ (via 
\eqref{k1invertp}), and the $G$-action on $L_{K(1)}
K(S[1/p])$ is nilpotent by \cite[Th.~5.6]{CMNN} (cf.~also
\cite[Lem.~4.20]{CM19}). 
\end{proof} 

\section{The truncating property of $L_{K(1)} \TR(-)$}

In this section, we prove the following basic result. 
Throughout, we fix a connective, $K(1)$-acyclic $E_\infty$-ring $B$, e.g.,
$H\mathbb{Z}$.

\begin{theorem} 
\label{truncatingthm}
The construction $L_{K(1)} \TR(-)$ is truncating on connective
$E_1$-$B$-algebras.
More generally, for any set $Q$,\footnote{This states informally that if $A$ is
a connective $B$-algebra, then the fiber of the map $\TR(A; \mathbb{F}_p) \to \TR(H\pi_0(A);
\mathbb{F}_p)$
has the property that each degree is annihilated by a power of $v_1$ 
(depending on the degree; note that this condition is slightly stronger than the
fiber simply being $K(1)$-acyclic).} the construction $L_{K(1)} (\prod_Q \TR(-))$ is
truncating on connective $E_1$-$B$-algebras. 
\end{theorem} 

The proof of 
\Cref{truncatingthm} will rely on a $K(1)$-acyclicity criterion for 
cyclotomic spectra (\Cref{k1acyccrit}), which will use
some elementary estimates for exponents of nilpotence 
with respect to the Bott element $\beta$. 

In the following, we let $ku$ denote the connective topological $K$-theory
spectrum, so $\pi_*(ku) = \mathbb{Z}[\beta]$ with $|\beta| = 2$. 
Since $B$ is $K(1)$-acyclic, the 
associative ring spectrum $B \otimes ku/p$ is annihilated by a power of $\beta$. 
Our strategy is roughly based on bounding the exponents of nilpotence for $\beta$ 
in the fixed points $\THH(B)^{C_{p^n}} \otimes ku/p$ and in particular
showing that they are
$O(p^n)$.

Recall that $ku$ is complex-oriented, leading to the following result.

\begin{proposition} 
\label{kunnethM}
Let $M$ be a $ku$-module equipped with an $S^1$-action. Then for each $n \geq
1$, the natural map 
induces an equivalence
\begin{equation} \label{Mhs1cp}  M^{hS^1} \otimes_{ku^{BS^1}} ku^{BC_{p^n}} \xrightarrow{\sim} M^{hC_{p^n}}
. \end{equation}
\end{proposition} 
\begin{proof} 
Compare~\cite[Sec.~7.4]{MNN17}. 
In fact, via the projection formula, $M^{hC_{p^n}} \simeq (M \otimes_{ku}
ku^{S^1/C_{p^n}})^{hS^1}$.  Here
$ku^{S^1/C_{p^n}}$ denotes the $ku$-valued function spectrum of $S^1/C_{p^n}$ with the
corresponding $S^1$-action, and the tensor product is taken in $\fun(BS^1, \mod(ku))$. 
Let $V_n$ denote the one-dimensional complex representation
of $S^1$ where $z \in S^1$ acts by multiplication by $z^{p^n}$, and let $S(V_n)$
denote the unit circle in $V_n$ as an $S^1$-space, so $S(V_n) \simeq
S^1/C_{p^n}$. 
The Spanier--Whitehead dual of the Euler sequence 
in $\fun(BS^1, \sp)$,
$S(V_n)_+ \to S^0 \to S^{V_n} $ 
and the complex-orientability of $ku$ together show that 
$ku^{S^1/C_{p^n}} \in \fun(BS^1, \mod(ku))$ belongs to the thick subcategory
generated by the unit. 
Therefore, applying the right adjoint
$(-)^{hS^1} \colon \fun(BS^1, \mod(ku)) \to \mod(ku^{BS^1})$, we find that the natural map 
\[ M^{hS^1} \otimes_{ku^{BS^1}} (ku^{S^1/C_{p^n}})^{hS^1} \to (M \otimes_{ku}
ku^{S^1/C_{p^n}})^{hS^1} = M^{hC_{p^n}}  \]
is an equivalence, whence the result. 
\end{proof} 

By complex orientability, we have
\[ \pi_*( ku^{BS^1}) = \mathbb{Z}[\beta][[x]], \quad |\beta| = 2, |x|  = -2
.\]
Consider the formal group law
over $\mathbb{Z}[\beta]$ given by $f(u,v) = u + v + \beta uv$; this is the
formal group law associated to the complex oriented ring spectrum $ku$, and is
homogeneous of degree $-2$ if $|u|, |v| = -2$. 
We have an equivalence
$ku^{BS^1}/x \simeq ku$. 
More generally, let $[p^n](x) \in \pi_* ( ku^{BS^1})$ denote the $p$-series of the formal group
law $f$; modulo $p$, we have 
$[p^n](x) \equiv \beta^{p^n-1} x^{p^n}$ since $[p^n] (x)$ is homogeneous of
degree $-2$ and recovers the multiplicative formal group law under the
specialization $\beta \mapsto 1$. 
By the Eilenberg--Moore spectral sequence or Gysin sequence (of which
\eqref{Mhs1cp} is a form), we have
\begin{equation} \label{kuBCp} ku^{BC_{p^n}} = ku^{BS^1}/
([p^n](x)).\end{equation}

\begin{proposition} 
\label{easybetabound}
Let 
$A$ be a connective $E_\infty$-ring spectrum with $S^1$-action. 
Suppose that $\beta^r = 0$ in $\pi_*(A \otimes ku/p)$. 
Then in $\pi_*((A \otimes ku)^{hC_{p^n}}/p)$, we have $\beta^{p^{n} -1 + rp^n} =
0$. 
\end{proposition} 
\begin{proof} 
Let $R = A \otimes ku/p$, so $R$ is an associative $ku$-algebra spectrum with
$S^1$-action. 
Consider the $S^1$-homotopy fixed points $R^{hS^1}$, which is a
$ku^{BS^1}$-algebra. 
Since $p = 0$ in $\pi_0 R$, 
we have
by \Cref{kunnethM} and \eqref{kuBCp},
\[ R^{hC_{p^n}} = R^{hS^1} \otimes_{ku^{BS^1}} ku^{BC_{p^n}} = 
R^{hS^1}/ [p^n](x) = R^{hS^1}/ (\beta^{p^n-1}x^{p^n}).\]

Our assumption is that $\beta^r =0$ in $\pi_*(R)$, which means that we can write
$\beta^r = x v \in \pi_*(R^{hS^1})$ for some $v$ since $R = R^{hS^1}/x$. 
It follows that in $\pi_*(R^{hC_{p^n}})$, we have
$\beta^{p^n -1 + rp^n} = 
\beta^{p^n-1} (xv)^{p^n} = 0$. 
\end{proof}

Now we apply the above to the $E_\infty$-ring $\THH(B)$ equipped with its
$S^1$-action; recall that $B$ is assumed connective and $K(1)$-acyclic. 
\begin{proposition} 
\label{betanilpbound}
There exists a constant $\kappa >0$ such that for each $i > 0$: 
\begin{enumerate}
\item  
We have
that $\beta^{ \kappa i} =0$ in $(ku/p \otimes \THH(B))^{hC_{i}}$.
\item
Let $N$ be a $\THH(B) \otimes ku$-module in $\fun(BS^1, \sp)$. 
Suppose that $N$ (as a $\THH(B) \otimes ku$-module in $\fun(BS^1, \sp)$) is induced from
the cyclic group $C_i \subset S^1$. 
Then for any $t$, $N_{hC_t}/p $ is annihilated by $\beta^{\kappa i}$. 
\end{enumerate}
\end{proposition} 
\begin{proof} 
To prove (1), 
using the transfer and restricting to a $p$-Sylow subgroup, we may reduce to the case
where $i$ is a power of $p$, say $i = p^n$. 
Then the claim follows from \Cref{easybetabound}, since $\beta$ is nilpotent in $ku/p
\otimes \THH(B)$ since $B$ is $K(1)$-acyclic. 
It follows that we can find a $\kappa_1$ such that $\beta^{\kappa_1 i} = 0$ in
$(ku/p \otimes \THH(B))^{hC_i}$ for all $i  > 0$. 

For (2), consider a $\THH(B) \otimes ku$-module $N$ in $\fun(BS^1, \sp)$ which is the induction
of 
a $\THH(B) \otimes ku$-module 
$N'$  in $ \fun(BC_{i}, \sp)$. 
It follows that $N_{hS^1} /p = (N' /p)_{hC_i} $ is a module over 
$(ku/p \otimes \THH(B))^{hC_i}$, writing homotopy orbits as a module over
homotopy fixed points; this is therefore annihilated by $\beta^{\kappa_1 i}$ by
the previous part of the result. 
For any $t$, let $W_t$ be the one-dimensional complex representation of $S^1$ where $z$ acts
by multiplication by $z^t$, so the unit circle $S(W_t)_+ = (S^1/C_t)_+$ as
$S^1$-spaces.
We have $N_{hC_t} = (N \otimes S(W_t)_+)_{hS^1}$ by the projection formula for
induction and restriction along $C_t \subset S^1$. 
Then the Euler sequence $S(W_t)_+ \to S^0 \to S^{W_t}$ (as
in the proof of \Cref{kunnethM}) and the complex-orientability of $ku$
yields a fiber sequence
$N_{h C_t}  \to N_{hS^1}  \to \Sigma^2 N_{hS^1} $. 
Therefore, $N_{hC_t} /p$ is annihilated by $\beta^{2 \kappa_1 i}$;
taking $\kappa = 2\kappa_1$ we conclude. 
\end{proof}

For the next result, we use the notion of a (nonnegatively) \emph{graded cyclotomic spectrum},
cf.~\cite[Sec.~3 and App.~A]{AMNN} or \cite{Bru01}. 
A graded cyclotomic spectrum 
consists of a 
graded spectrum $X = \bigoplus_{i \geq 0} X_i$ equipped with an $S^1$-action together
with a graded $S^1$-equivariant cyclotomic
Frobenius $\varphi_i \colon  X_i \to X_{pi}^{tC_p}$ for each $i$.  
Given a nonnegatively graded $E_1$-ring $R$, the topological Hochschild homology
$\THH(R)$ acquires the structure of a graded cyclotomic spectrum. 
Given a graded cyclotomic spectrum $X$, we can consider a graded
cyclotomic spectrum $X_{\geq i}$ where we only consider the graded summands in
degrees $\geq i$; this gives any graded cyclotomic spectrum a natural descending
filtration. 
The filtration quotients $X_{\geq i}/X_{\geq pi}$ have trivialized Frobenius
because of the grading, and their $\TR$ can be thus described explicitly: 

\begin{construction}[$\TR$ of cyclotomic spectra with zero Frobenius] 
Suppose $X \in \cycsp_{\geq 0}$ is $p$-complete with trivialized Frobenius. Then, 
as in \cite[Rem.~2.5]{AN} and \cite[Cor.~II.4.7]{NS18},
we obtain
a natural equivalence
\begin{equation} \label{TRFrobzero}  \TR(X; \mathbb{Z}_p) \simeq \prod_{i \geq 0} X_{hC_{p^i}}.  \end{equation}
For this, we also use the identification for $n \geq 1$,
$(X^{tC_p})^{hC_{p^{n-1}}} \simeq X^{tC_{p^n}}$, cf.~\cite[Lem.~II.4.1]{NS18}. 
\end{construction}

\begin{proposition}[$K(1)$-acyclicity criterion]
\label{k1acyccrit}
Let $X$ be a positively graded cyclotomic spectrum with the structure of
$\THH(B)$-module. Suppose for each $i > 0$: 
\begin{enumerate}
\item  
$X_i$ is $ (2p+ 1) \kappa i $-connected (where $\kappa$ is as in \Cref{betanilpbound}).  
\item
As a $\THH(B)$-module in  $\fun(BS^1, \sp)$, $X_i$ is induced from the
cyclic group $C_i \subset S^1$. 
\end{enumerate}
Then for any set $Q$, $\prod_Q \TR(X)$ is $K(1)$-acyclic. 
\end{proposition} 

\begin{proof} 
For simplicity of notation, we write $\TR_Q(-) = \prod_Q \TR(-)$. 
The construction $\TR_Q(-)$ is exact and commutes with geometric realizations on
$\cycsp_{\geq 0}$; therefore, it commutes with tensoring with $ku$. Without
loss of generality, we can therefore 
assume that $X$ is a $ \THH(B) \otimes ku$-module in graded cyclotomic spectra. Here we
regard $ku$ as a trivial cyclotomic spectrum, i.e., via the image of the unique
symmetric monoidal functor $\sp \to \cycsp$. 

We consider the descending filtration $\{ \fil^{\geq n } X = X_{\geq p^{n}}\}_{n \geq 0}$
on the cyclotomic spectrum $X$; note that the associated graded terms have
trivialized Frobenii for degree reasons. 
This yields a filtration on the spectrum $\TR_Q(X)$, with 
$\fil^{\geq n} \TR_Q(X) = \TR_Q(X_{\geq p^n})$. 
Since $\TR(-)$ preserves connectivity, our assumptions imply that 
$\fil^{\geq n} \TR_Q(X)$ is $(2p + 1) \kappa p^n$-connective. 

We have
by \eqref{TRFrobzero},
$\gr^n \TR_Q(X) =  \bigoplus_{p^{n} \leq i < p^{n+1}} \prod_Q\prod_{t \geq 0}
(X_i)_{hC_{p^t}}$, since the cyclotomic Frobenius is trivial on $\gr^n X$ for
grading reasons.
Since $X_i$ is induced from $C_i \subset S^1$, 
it follows from \Cref{betanilpbound} that 
the $ku$-module
$\gr^n \TR_Q(X)/p $ is annihilated by 
$\beta^{ \kappa p^{n+1}}$.

Now $|\beta| = 2$, and 
$2\sum_{i = 0}^{n-1} \kappa p^{i+1} < (2p+1)\kappa p^{n}$ with the difference tending to
$\infty$ as $n \to \infty$. 
Applying \Cref{invertingiszero} below, we  conclude that 
inverting $\beta$ 
annihilates
$\TR_Q(X) /p$, whence $\TR_Q(X)$ is $K(1)$-acyclic as desired. 
\end{proof}

\begin{lemma} 
Let $R$ be a connective $E_1$-ring spectrum, and let $x \in \pi_t(R)$ be an
element. 
Let $\left\{\fil^{\geq n} Y\right\}_{n \geq 0}$ be a 
filtered $R$-module spectrum. 
Suppose that there exist functions $f, g\colon  \mathbb{N} \to \mathbb{N}$ such that: 
\begin{enumerate}
\item $\gr^n (Y)$ is annihilated by $x^{f(n)}$. 
\item  $\fil^{\geq n}(Y)$ is $g(n)$-connective. 
\item  $g(n) - t\sum_{i = 0}^{n-1} f(i)  \to \infty $ for $n  \to \infty$. 
\end{enumerate}
Then $Y[1/x] = 0$. 
\label{invertingiszero}
\end{lemma} 
\begin{proof} 
Let $y \in \pi_s(Y)$. 
For each $n  > 0$, then the class
$x^{f(0) + f(1) + \dots + f(n-1)} y  \in \pi_*(Y)$ 
naturally lifts to $\pi_{s + t(f(0) + \dots  + f(n-1))}( \fil^{\geq n} Y)$. 
But for $n \gg 0$, 
the connectivity of $\fil^{\geq n} Y$
forces this last group to vanish. 
Therefore, the image of $y$ in $\pi_s(Y[1/x]) $ vanishes as desired. 
\end{proof}

We now use the following basic calculation of $\THH$ of a free associative
algebra, as a spectrum equipped with $S^1$-action. 
Versions of this are classical in ordinary Hochschild homology,
cf.~\cite[Sec.~3.1]{Loday}. 
In the language of factorization homology, this result is a special case
of the calculation of the factorization homology of a free algebra,
\cite[Prop.~4.13]{AFT} and \cite[Prop.~5.5]{AF15}. 

\begin{theorem}[$\THH$ of a free associative algebra] 
Let $M$ be a spectrum, and let $T(M) = \bigoplus_{n \geq 0}M^{\otimes n}$ be the
free $E_1$-algebra spectrum
generated by $M$. 
Then there is a natural equivalence
in $\fun(BS^1, \sp)$,
\begin{equation} \label{THHtensor}  \THH( T(M)) \simeq \bigoplus_{n \geq 0} \mathrm{Ind}_{C_n}^{S^1} (M^{\otimes
n}) , \end{equation}
where we use the natural $C_n$-action on $M^{\otimes n}$ by permuting the
factors. 
\end{theorem} 
\begin{proof} 
The results of \emph{loc.~cit.}~(applied to the 
framed manifold $S^1$) imply that 
$\THH( T(M)) \simeq \int_{S^1} T(M) = \bigoplus_{n \geq 0} (\mathrm{Conf}_n(S^1)_+ \otimes M^{\otimes
n})_{h \Sigma_n}$, for $\mathrm{Conf}_n(S^1)$ the configuration space of $n$
ordered points on the circle. 
One checks now (cf.~\cite[Ex.~II.14.4]{CJ98}) that $\mathrm{Conf}_n(S^1)$, as a
space with $S^1 \times \Sigma_n$-action, 
is homotopy equivalent to $(S^1 \times  \Sigma_n)/C_n$ (with $C_n$ embedded
diagonally), whence the claim. 
\end{proof}

\begin{proposition} 
\label{lK1THHTR}
Let $M, N$ be connective spectra. 
Then the map of cyclotomic spectra
\[ \THH( T(M \oplus N)) \otimes \THH(B) \to \THH(T(M)) \otimes
\THH(B)  \]
induces an equivalence on $L_{K(1)} ( \prod_Q \TR(-))$ for any set $Q$ if $N$ is $\geq \kappa(2p+1)$-connective. 
\end{proposition} 
\begin{proof} 
We can consider the tensor algebra $T(M \oplus N)$ as a graded $E_1$-ring
spectrum where $M$ is placed in degree zero and $N$ is placed in degree $1$. 
In this case, 
if we collect the terms in \eqref{THHtensor}, 
we find that the $i$th graded piece of $\THH( T(M \oplus N))$ is the component
which is $i$-homogeneous. 
Explicitly, for any subset $ I \subset \left \langle n\right\rangle = \left\{1, 2, \dots,
n\right\}$, we write 
$(M, N)^{\otimes ( \left \langle n \right\rangle\setminus I, I)}$ for the ordered tensor
product of $n$ factors, where the $j$th factor is $M$ if $j \notin I$ and $N$ if
$j \in I$. 
Expanding \eqref{THHtensor} gives
\begin{equation} \label{THHexpand} \THH( T(M \oplus N))_i = \bigoplus_{n \geq 0}
\mathrm{Ind}_{C_n}^{S^1} \left(  \bigoplus_{I \subset \left \langle
n \right\rangle, |I| = i} (M, N)^{\otimes (\left \langle n \right\rangle\setminus I,
I)} \right). \end{equation} 
Here the $C_n$-action on the parenthesized term in \eqref{THHexpand} permutes
the various summands. 
Note in particular that  the stabilizer of $I \subset \left \langle
n\right\rangle$ in the
$n$th summand is a subgroup of a cyclic group $C_i \subset
C_n$ since $|I| = i$. 
In particular, it follows that the $i$th graded piece of $\THH( T(M \oplus N))$
is induced from $C_i \subset S^1$. 
Furthermore, since $N$ is $\geq \kappa(2p+1)$-connective, it follows that the $i$th
graded piece of $\THH( T( M \oplus N)) $ is $\geq \kappa(2p+1)i$-connective. 
By 
\Cref{k1acyccrit}, it follows that the positively graded part of $\THH(T(M
\oplus N)) \otimes \THH(B)$ has vanishing $K(1)$-local $\prod_Q \TR(-)$ for any
set $Q$, whence
the result. 
\end{proof} 

For the next result, 
if $R$ is an $E_1$-ring spectrum and $N$ an $(R, R)$-bimodule, we 
let $T_R(N) = \bigoplus_{n \geq 0} N \otimes_R N \otimes_R \dots \otimes_R N$ be
the free $E_1$-algebra under $R$ generated by $N$. 

\begin{proposition} 
\label{lK1TR2}
Let $R$ be a connective $B$-algebra. 
Let $N$ be a $\geq \kappa(2p+1)$-connective $(R, R)$-bimodule. 
Then the map 
$\THH( T_R(N)) \to \THH(R)$ induces an equivalence on 
$L_{K(1)} (\prod_Q \TR(-))$ for any set $Q$. 
\end{proposition} 
\begin{proof} 
Simplicially resolving $N$ by free $(R, R)$-bimodules (since $\prod_Q \TR(-)$ commutes with
geometric realizations), 
we may assume that $N$ is free on generators (possibly infinitely many) in
degrees $\geq \kappa(2p+1)$. 
Simplicially resolving $R$ by free $B$-algebras, we may assume that $R$ is free as
well, on some classes in degrees $\geq 0$. 
In this case, the result follows from 
\Cref{lK1THHTR}.
\end{proof}

\begin{lemma} 
\label{ktruncatingimpliestruncating}
Let $E$ be a localizing invariant of $B$-linear $\infty$-categories. 
Suppose that there exists $k \geq 0$ such that for every connective $B$-algebra $R$, we have 
$E( R) \xrightarrow{\sim} E( \tau_{\leq k} R)$. 
Then $E$ is truncating. 
\end{lemma} 
\begin{proof}
We show by descending induction\footnote{This type of argument is also used
in \cite[Sec.~3.2]{LMT}.} that for any $i \geq 0$ and for any connective
$B$-algebra $R$, the map 
$R \to \tau_{\leq i} R$
induces an equivalence on $E$; taking $i = 0$ gives the theorem. 
For $i \geq k$, we already know the claim by assumption. 
Suppose we know the claim for $i+1$; to prove the claim for $i$, we need to see
that $\tau_{\leq i+1} R \to \tau_{\leq i} R $ induces an equivalence on
$E$. 
Now $\tau_{\leq i+1} R \to \tau_{\leq i} R$ is a square-zero extension, i.e., we
have a pullback diagram of $B$-algebras
\[ \xymatrix{
\tau_{\leq i + 1} R \ar[d]  \ar[r] & H(\pi_0 R) \ar[d]  \\
\tau_{\leq i } R \ar[r] &  H(\pi_0 R) \oplus H ( \pi_{i+1} R)[i+2]
}.\]
Since $E$ is a localizing invariant, 
the main result of Land--Tamme \cite{LT19} 
yields an $B$-algebra $\widetilde{R}$ with underlying spectrum 
$H(\pi_0 R) \otimes_{\tau_{\leq i+1} R } \tau_{\leq i} R$ fitting into a
commutative diagram
of $B$-algebras
\[  \xymatrix{
\tau_{\leq i+1} R \ar[d]  \ar[r] &  H(\pi_0 R) \ar[d]  \\
\tau_{\leq i} R \ar[r] &  \widetilde{R}
}\]
which is carried to a pullback by $E$. 
But the map $H(\pi_0 R) \to \widetilde{R}$ is an equivalence in degrees $\leq
i+1$ and therefore induces an equivalence on $E$ by the inductive
hypothesis. 
Therefore, $E( \tau_{\leq i+1} R) \to E( \tau_{\leq i} R)$
is an equivalence, whence the inductive step and the result. 
\end{proof} 

\begin{proof}[Proof of \Cref{truncatingthm}] 
As before, we write $\TR_Q(-) = \prod_Q \TR(-)$ for a set $Q$. 
For any connective $B$-algebra $R$, we claim that
$R \to \tau_{\leq \kappa(2p+1)} R$ induces an equivalence on $L_{K(1)}
\TR_Q(-)$. 
Indeed, this follows from \Cref{lK1TR2} because we can simplicially resolve
$\tau_{\leq \kappa(2p+1)} R$
using free $R$-algebras over free $(R, R)$-bimodules on classes in degrees $\geq
\kappa(2p+1)
$ and since $\TR_Q(-)$ commutes with geometric realizations. 
The result now follows from \Cref{ktruncatingimpliestruncating}. 
\end{proof}

\begin{corollary}[Cf.~\cite{LMT}] 
The invariants $L_{K(1)} K(-), L_{K(1)} \TC(-)$ are truncating on connective $B$-algebras. 
\label{LMTtruncating}
\end{corollary} 
\begin{proof} 
The result for $L_{K(1)} \TC(-)$ follows from \Cref{truncatingthm} by taking
Frobenius fixed points. 
The result for $L_{K(1)} K(-)$ is a formal consequence 
since the fiber of the trace $K(-) \to \TC(-)$
is truncating by the 
Dundas--Goodwillie--McCarthy theorem \cite{DGM}. 
\end{proof}

\begin{corollary}[\cite{LMT, BCM}] 
For any $n$, we have $L_{K(1)} K(\mathbb{Z}/p^n) = 0$. 
\end{corollary} 
\begin{proof} 
Take $B = H\mathbb{Z}$ in \Cref{truncatingthm}, so $L_{K(1)} \TC(-)$ is truncating
and therefore nilinvariant on connective $H\mathbb{Z}$-algebras. 
Then the result follows because the Dundas--Goodwillie--McCarthy theorem and
comparison with $\mathbb{F}_p$ yields $L_{K(1)}
K(\mathbb{Z}/p^n) = L_{K(1)} \TC(\mathbb{Z}/p^n)$, but the above shows that this
equals
$L_{K(1)} \TC(\mathbb{F}_p) = 0$. 
\end{proof}

\section{Asymptotic $K(1)$-locality}

In this section, we 
show (\Cref{asymptoticK1TR})
that  $\TR$ is asymptotically $K(1)$-local for a regular
ring satisfying  mild hypotheses, using the
Beilinson--Lichtenbaum conjecture. 
This result is due to Hesselholt--Madsen in the case of smooth algebras over a
DVR with perfect residue field of characteristic $>2$, which we begin by
reviewing.

\begin{theorem}[Hesselholt--Madsen \cite{HM03, HM04}] 
\label{HMasymptotic}
Let $K$ be a complete, discretely valued field of characteristic zero 
with ring of integers $\mathcal{O}_K \subset K$  and perfect residue field $k$ of characteristic $p > 2$.
Let $R$ be a smooth $\mathcal{O}_K$-algebra of relative dimension $d$. 
Then the map
\[  \TR(R; \mathbb{F}_p) \to L_{K(1)} \TR(R; \mathbb{F}_p)   \]
is $d$-truncated. 
\end{theorem} 
\begin{proof} 
Without loss of generality, we can assume that $\mu_p \subset K$, since
otherwise $\TR(R; \mathbb{F}_p)$ is a retract of $\TR(R[\zeta_p];
\mathbb{F}_p)$ via the transfer. 
In this case, the result follows from 
\cite[Th.~E]{HM04}. 
Indeed, 
\emph{loc.~cit.}~gives a calculation of 
the cofiber $\TR(R|R_K; \mathbb{F}_p)$ of the transfer map $\TR(R \otimes_{\mathcal{O}_K} k; \mathbb{F}_p) \to \TR(R;
\mathbb{F}_p)$. 
Since $\TR(R \otimes_{\mathcal{O}_K} k; \mathbb{F}_p)$ is $K(1)$-acyclic and
$d$-truncated
in view of the identification \cite{Hes96} with the de Rham--Witt complex of $R
\otimes_{\mathcal{O}_K} k$, 
it suffices to 
verify the (stronger) claim that 
$\TR(R|R_K; \mathbb{F}_p) \to 
L_{K(1)}\TR(R|R_K; \mathbb{F}_p)
$ is $\leq  d-2$-truncated. 
Equivalently (cf.~\Cref{K1locallemma} below), it suffices to show that the cofiber of the Bott element on 
$\TR(R|R_K; \mathbb{F}_p)$ is 
$\leq d+1$-truncated. 
In fact, this follows from the calculation in 
\emph{loc.~cit.}, 
once we know that the absolute de Rham--Witt complex
$W \Omega^{\ast}_{(R, M_R)}$ is $p$-divisible in degrees $\geq d+2$. 
This in turn follows from the case
where $R = \mathcal{O}_K$, cf.~\cite[Cor.~3.2.7]{HM03}; the case of polynomial
rings 
via the functor $P$, cf.~\cite[Lem.~7.1.4]{HM04} and its proof; and 
finally \'etale base-change \cite[Lem~7.1.1]{HM04}. 
\end{proof} 

\begin{construction}[$\TR$ as $p$-typical curves] 
Let $R$ be an animated ring.\footnote{Also called a simplicial commutative
ring; cf.~\cite{CS} for a discussion of this terminology.} 
We let \begin{equation} \label{curvesR} C(R) = \varprojlim_n \Omega K( R[x]/x^n, (x))
\end{equation} denote the spectrum of
\emph{curves} on the $K$-theory of $R$, defining a functor from animated
rings to spectra. 
By \cite[Th.~3.1.9]{Hes96}, 
if $R$ is a discrete $\mathbb{Z}/p^j$-algebra for some $j$, we have
a natural 
expression of $\TR(R; \mathbb{Z}_p)$ as a summand of $C(R; \mathbb{Z}_p)$ (note
that we are considering $p$-typical $\TR$, while in \emph{loc.~cit.} global
$\TR$ is considered). 
Left Kan extending both sides  to animated rings (since both $\TR(-)$ and $C(-)$ commute with
geometric realizations\footnote{For $C(-)$, this follows from the expression
$K(R[x]/x^n, (x)) = \TC(R[x]/x^n, (x))$, using that $\TC$ of connective 
ring spectra commutes with geometric realizations, and passing to the limit
along $n$.}), we obtain that $\TR(R; \mathbb{Z}_p) $ is naturally a
summand of $C(R)$ for $R$ an animated $\mathbb{Z}/p^j$-algebra. 
Passing to the limit over $j$ and using the $p$-adic continuity of $K$-theory
\cite[Th.~5.21]{CMM}, we find that 
for any $p$-henselian animated ring
$R$ (i.e., $\pi_0(R)$ is $p$-henselian), 
$\TR(R; \mathbb{Z}_p)$ is a summand of $C(R; \mathbb{Z}_p)$. 
\label{ptypiccurves}
\end{construction}

Next, we discuss the comparison between the $p$-adic $K$-theory of $R$ and
$R[1/p]$, cf.~\cite[Lem.~3.5]{Niz08} for this argument. 
\begin{proposition} 
\label{invertingpconn}
Let $R$ be a regular noetherian ring of finite Krull dimension. 
Suppose for every $x \in \spec(R/p)$, we have $[\kappa(x): \kappa(x)^p] \leq
p^d$. 
Then the map $K(R; \mathbb{F}_p) \to K(R[1/p]; \mathbb{F}_p)$ is 
$d$-truncated. 
\end{proposition} 
\begin{proof} 
The homotopy fiber of the map in question is, by Quillen's d\'evissage theorem,
the mod $p$ $K$-theory of the abelian category of finitely generated
$R/p$-modules, 
i.e., the mod $p$ $G$-theory of $R/p$, $G(R/p; \mathbb{F}_p)$. 
So it suffices to show 
that $G(R/p; \mathbb{F}_p)$ is concentrated 
in degrees $\leq d$. 
Using the 
filtration by codimension of support and d\'evissage
(cf.~\cite[Th.~5.4]{Qui72}), 
we find that 
$G(R/p; \mathbb{F}_p)$ has a filtration whose associated graded terms are direct
sums of the $K( \kappa(x); \mathbb{F}_p)$ for $x \in \spec(R/p)$; it therefore
suffices to show that these terms are $d$-truncated. 
But now the Geisser--Levine theorem \cite{GL00} 
implies that for any field $E$ of characteristic $p$, there is an embedding
$K_*(E; \mathbb{F}_p) \hookrightarrow \Omega^{\ast}_{E/\mathbb{F}_p}$. 
This implies that $K(\kappa(x); \mathbb{F}_p)$ is 
$\dim_{\kappa(x)} \Omega^1_{\kappa(x)/\mathbb{F}_p}$-truncated. 
But 
$
\dim_{\kappa(x)} \Omega^1_{\kappa(x)/\mathbb{F}_p} = 
\log_p [\kappa(x): \kappa(x)^p]  \leq d
$ by the theory of $p$-bases
\cite[Th.~26.5]{Mat89}. Combining these facts, the result follows. 
\end{proof}

\begin{proposition}[Rosenschon--{\O}stv{\ae}r \cite{RO06}]
\label{asymptoticKlocalityRO}
Let $R$ be a regular noetherian $\mathbb{Z}[1/p]$-algebra of finite Krull dimension. 
Suppose that $\mathrm{vcd}_p(\kappa(x)) \leq d$ for all $x \in \spec(R[1/p])$. 
Then the map 
$K(R; \mathbb{F}_p) \to L_{K(1)} K(R; \mathbb{F}_p)$ 
has $\leq d-3$-truncated homotopy fiber. 
\end{proposition} 
\begin{proof}[Proof sketch] 
Using Nisnevich descent \cite{TT90}, we reduce to the case where $R$ is
henselian local, and even
a field of characteristic $\neq p$ by Gabber--Suslin rigidity \cite{Gabber92}; note that we
do not have to worry about the distinction between connective and nonconnective
$K$-theory by regularity. Then, the result follows from 
the Beilinson--Lichtenbaum conjecture (proved by Voevodsky--Rost,
cf.~\cite{HW19}) describing the
associated graded terms of the motivic filtration on $K(R; \mathbb{F}_p)$, see
e.g., \cite[Sec.~6.2]{CM19} for an account. 
\end{proof}

\begin{proposition} 
\label{pcohdimbound}
Let $R$ be an excellent normal domain which is henselian along an
ideal $I \subset R$ containing $(p)$. 
Suppose that for all $\mathfrak{p} \in \spec(R)$ containing $I$, we have $\dim R_{\mathfrak{p}}  +
\log_p [\kappa(\mathfrak{p}): \kappa(\mathfrak{p})^p] \leq d$. 
Then 
for all $\mathfrak{q} \in \spec(R[1/p])$, the residue field
$\kappa(\mathfrak{q})$ has $p$-cohomological dimension at most $  d + 1$. 
\end{proposition} 
\begin{proof} 
By standard continuity arguments, it suffices to show that for any affine open $U = \spec( R[1/f]) \subset
\spec(R[1/p])$ and any constructible $p$-torsion sheaf $\mathcal{F}$ on $U$,  we
have
$H^n( U, \mathcal{F}) = 0$ for $n >  d +1$. 
Denote by $j\colon  U \hookrightarrow \spec(R)$ the open inclusion, and denote by $i_0\colon 
\spec(R/I) \hookrightarrow \spec(R)$ the closed embedding. 
Using that $\spec(R/I)$ has $p$-cohomological dimension $\leq 1$ and the  affine
analog of proper base change \cite{Ga94, Hu93} applied to the henselian ideal
$I \subset R$, we see that it suffices
to show that $i_0^* R^n j_* \mathcal{F} = 0$ for $n >   d$. 

Working stalkwise on $R$ and using the compatibility of \'etale cohomology with
filtered colimits, we can now reduce to the case where $R$ is an excellent,  strictly
henselian normal local domain with residue field $k$ (and $I \subset R$ is
contained in the maximal ideal), since excellence and normality are preserved by
strict henselization (cf.~\cite{Gre76} for the former). 
The statement then becomes that $H^n( U, \mathcal{F}) = 0$ for $n >  
d$, which follows from the bound of Gabber--Orgogozo \cite[Th.~6.1]{GO08} (noting that the $p$-dimension of
the residue field $k$ is $\log_p [k:k^p]$ since $k$ is separably closed). 
\end{proof}

\begin{remark} 
\label{dimremark}
Suppose that 
$R_0$ is a ring such that $R_0/p$ is finite type over a perfect field $k$ of characteristic
$p$. Then for every prime ideal $\mathfrak{p} \in \spec(R_0)$ containing $(p)$, we have
$\dim (R_0)_{\mathfrak{p}} + \log_p [\kappa(\mathfrak{p}): \kappa(\mathfrak{p})^p]
\leq \dim(R_0)$.  
Indeed, we have 
$\dim (R_0)_{\mathfrak{p}} + \dim (R_0/\mathfrak{p}) \leq \dim(R_0)$. 
Therefore, it suffices to prove 
$\log_p [ \kappa( \mathfrak{p}): \kappa(\mathfrak{p})^p] =
\dim(R_0/\mathfrak{p})$. 
But in this case, both are the transcendence degree of the field $\kappa(
\mathfrak{p}) $ over $k$, cf.~\cite[Th.~27.B]{Mat80}. \end{remark} 

\begin{lemma} 
\label{powerserieslem}
Let $R_0$ be an $\mathbb{F}_p$-algebra. 
Suppose that 
for all residue fields $\kappa$ of $R_0$, we have $[\kappa: \kappa^p] \leq p^d $
for some $d \geq 0$. 
Then for all residue fields $\kappa'$ of $R_0[[x]]$, we have $[\kappa':
\kappa'^p] \leq p^{d+1}$. 
\end{lemma} 
\begin{proof} 
We have that $R_0[[x]]$ is generated by $p$ elements over the subring generated
by its $p$th powers and by $R_0$. For every residue field $\kappa$ of $R_0$, 
it follows that $R_0[[x]] \otimes_{R_0} \kappa$ is generated by $p$ elements
over the subring generated by $p$th powers together with $\kappa$. 
In particular, it is generated by $p^{d+1}$ elements as a module over its $p$th
powers. 
It follows that the same holds for any residue field of $R_0[[x]] \otimes_{R_0}
\kappa$ and, varying $\kappa$, we obtain the conclusion for every residue field
of $R_0$. 
\end{proof}

\begin{theorem} 
\label{asymptoticK1TR}
Let $R$ be an excellent,  $p$-torsionfree regular noetherian ring.   Suppose 
$R/p$ is finitely generated as a module over its subring of $p$th powers. 
Suppose furthermore that
for all $\mathfrak{p} \in \spec(R)$ containing $(p)$, we have
$\dim R_{\mathfrak{p}} + \log_p [ \kappa(\mathfrak{p}):\kappa(\mathfrak{p})^p]
\leq d$ for some $d \geq 0$. 
Then the map $\TR(R; \mathbb{F}_p) \to L_{K(1)} \TR(R; \mathbb{F}_p)$ is 
$ (d-1)$-truncated. 
\end{theorem} 
\begin{proof} 
Without loss of generality, we can assume that $R$ is $p$-henselian (since
henselization preserves excellence, cf.~\cite{Gre76}), so $\dim R \leq d$. 
By 
Construction~\ref{ptypiccurves}, it suffices to show that 
$C(R)/p \to L_{K(1)} C(R)/p$ is 
$(d-1)$-truncated. Now $C(R)$ is the desuspension of the fiber of the map 
$\varprojlim_n K( R[x]/x^n) \to K(R)$. 
With $p$-adic coefficients, the fact that $R/p$ is finitely generated as a
module over its
$p$th powers allows us to pass to the
limit \cite[Th.~F]{CMM} and we obtain
\[ C(R; \mathbb{F}_p) = \Omega  K(R[[x]], (x); \mathbb{F}_p).  \]
Since $K(R)$ is a retract of $K(R[[x]])$, it suffices to show that the fiber of 
$K(R[[x]]; \mathbb{F}_p) \to L_{K(1)}K(R[[x]]; \mathbb{F}_p) $ 
is $d$-truncated. 
We will verify this by comparing both sides with the intermediate term $K( R[[x]][1/p];
\mathbb{F}_p)$.

First, 
for every characteristic $p$ residue field $\kappa$ of $R$, corresponding to a prime
ideal $\mathfrak{p} \subset R$ containing $(p)$,  we have 
$[\kappa: \kappa^p] \leq p^{d-1}$ by our assumption, since $\dim
R_{\mathfrak{p}} \geq 1$. 
Now
$R[[x]]$ is also a $p$-torsionfree regular ring of dimension $\dim(R) + 1$.
For every characteristic $p$ residue field $\kappa'$ of $R[[x]]$, we have 
$[\kappa': \kappa'^p] \leq p^{d}$ by \Cref{powerserieslem}. 
By 
\Cref{invertingpconn}, it follows that
$K(R[[x]]; \mathbb{F}_p) \to K(R[[x]][1/p]; \mathbb{F}_p)$ is 
$d$-truncated.

Second, we apply 
\Cref{pcohdimbound} to the ring $R[[x]]$ and the ideal $I = (p, x)$. 
The power series ring $R[[x]]$ remains excellent (and regular) since $R$ is excellent, thanks to 
\cite{KS19}. 
For any prime ideal $\mathfrak{p} \in \spec( R[[x]])$ containing $I$,  we 
let $\mathfrak{p}_0 = \mathfrak{p} \cap R \subset R$, so that
$R[[x]]_{\mathfrak{p}}/(x) = R_{\mathfrak{p}_0}$. 
Then 
$\dim( R[[x]]_{\mathfrak{p}}) 
 = \dim( R_{\mathfrak{p}_0})+ 1$ and 
 $\kappa( \mathfrak{p}) = \kappa(\mathfrak{p}_0)$. 
 Thus, 
\Cref{pcohdimbound} applies to $R[[x]]$ (with $d$ replaced by $d+1$) and 
we find that the characteristic zero residue fields of $R[[x]]$ have
$p$-cohomological dimension at most $d + 2$. 
Therefore, 
by
\Cref{asymptoticKlocalityRO}, 
$ K( R[[x]][1/p]; \mathbb{F}_p) \to L_{K(1)} K( R[[x]][1/p])$ is $
(d-1)$-truncated.

Combining the above, we find that the composite map 
$K(R[[x]]; \mathbb{F}_p) \to L_{K(1)} K(R[[x]]; \mathbb{F}_p) \simeq L_{K(1)}
 K( R[[x]][1/p];
\mathbb{F}_p)$ (the last identification by 
\eqref{k1invertp})
is $d$-truncated, whence the result. 
\end{proof} 

This recovers in particular \Cref{HMasymptotic}. 
More generally, we have: 
\begin{example} 
Let $R$ be a $d$-dimensional regular, excellent $p$-torsionfree noetherian ring with $R/p$ finitely generated over its
$p$th powers. 
Suppose that $(R/p)_{\mathrm{red}}$ is 
finite type over a perfect field (and necessarily of dimension $d-1$). 
Then \Cref{asymptoticK1TR} (in view of \Cref{dimremark}) applies to show that
$\TR(R; \mathbb{F}_p) \to L_{K(1)} \TR(R; \mathbb{F}_p)$ is $(d-1)$-truncated. 
\end{example}

\section{The Segal conjecture}

In this section, we discuss the relationship between the following two
properties of a cyclotomic spectrum $X$: 
\begin{enumerate}
\item  $\TR(X)$ agrees with its $K(1)$-localization in high enough degrees. 
\item The cyclotomic Frobenius $\varphi_X \colon  X \to X^{tC_p}$ is an equivalence in
high enough degrees. 
\end{enumerate}

The first property is the Lichtenbaum--Quillen style statement discussed in the
previous section, and verified for $\THH(R)$ under regularity and finiteness
hypotheses. 
The second property is often referred to as the ``Segal conjecture'' since for
$X = \THH(\mathbb{S})$, the Frobenius $\mathbb{S} \to \mathbb{S}^{tC_p}$ is a
$p$-adic 
equivalence by the Segal conjecture for $C_p$, proved in \cite{Lin80, Gun80}. 
The Segal conjecture has been studied extensively for $\THH(R)$ for $R$ a ring
(or ring spectrum).

We first show the implication $(1) \implies (2)$. 
We use the theory of topological Cartier modules of Antieau--Nikolaus \cite{AN},
which we begin by briefly reviewing.

\begin{definition}
A \emph{topological Cartier module} $M$ is an object of $\fun(BS^1, \sp)$ together with
maps $V \colon M_{hC_p} \to M$ and $F \colon M \to M^{hC_p}$ in
$\fun(BS^1, \sp)$ together with a
homotopy between the composite and the norm map $M_{hC_p} \to M^{hC_p}$
(considered $S^1 \simeq S^1/C_p$-equivariantly). 
The collection of topological Cartier modules is naturally organized into a
presentable stable $\infty$-category. 
\end{definition} 

Given a
bounded-below, $p$-typical cyclotomic spectrum $X$, we can  consider $\TR(X)$ as
a topological Cartier module, and we have an identification $X \simeq \mathrm{cofib}(V)$. 
Under these identifications, the cyclotomic Frobenius $X \to X^{tC_p}$ is obtained from 
$F\colon  \TR(X) \to \TR(X)^{hC_p}$ by taking cofibers by $V$ on both domain and
codomain and identifying $\TR(X)^{tC_p} \simeq X^{tC_p}$ as
$(\TR(X)_{h{C_p}})^{tC_p}
= 0$ \cite[Lem.~I.2.1]{NS18}. 
On bounded-below objects, this construction establishes a fully faithful 
embedding from cyclotomic spectra into topological Cartier modules, with image
given by the $V$-complete objects \cite[Th.~3.21]{AN}. 

\begin{proposition} 
\label{trunctosegal}
Let $X$ be a connective,  $p$-complete cyclotomic spectrum whose underlying
spectrum is $K(1)$-acyclic. 
Suppose the map $\TR(X) \to L_{K(1)} \TR(X) $ is $d$-truncated. Then the Frobenius
$\varphi\colon  X \to X^{tC_p}$ is $d$-truncated. 
\end{proposition} 
In the case $L_{K(1)} \TR(X) =0$, the result is \cite[Prop.~2.25]{AN}. 

\begin{proof} 
Since $X$ is $K(1)$-acyclic, it follows that 
$V\colon  \TR( X)_{hC_p} \to \TR(X)$ is $K(1)$-locally an equivalence. 
The $K(1)$-localization 
$L_{K(1)} \TR(X)$ acquires the structure of a topological Cartier
module as well by $K(1)$-localizing $F, V$ and using the comparison map 
$L_{K(1)} ( \TR(X)^{hC_p}) \to (L_{K(1)} \TR(X))^{hC_p}$. 
The composite map 
$(L_{K(1)} \TR(X))_{hC_p} \to (L_{K(1)} \TR(X))^{hC_p}$ is an equivalence
after $p$-completion since
Tate constructions vanish in the $K(1)$-local category. Since we saw that $V$
is an equivalence on $L_{K(1)} \TR(X)$ after $p$-completion, it follows that 
the Frobenius on $L_{K(1)} \TR(X)$
induces an equivalence
$$L_{K(1)} \TR( X) \xrightarrow{\sim} (L_{K(1)} \TR(X))^{hC_p}.$$

For a topological Cartier module $Y$, we consider the fiber of $F  = F_Y\colon  Y \to Y^{hC_p}$,
which we denote $\mathrm{fib}(F)$. 
As we saw above, $\mathrm{fib}(F)$ is contractible for the 
topological Cartier module $L_{K(1)} \TR(X)$. 
Moreover, $\mathrm{fib}(F)$ is $d$-truncated for the topological Cartier module
$\mathrm{fib}( \TR(X) \to L_{K(1)} \TR(X))$ because this topological Cartier
module is itself $d$-truncated. 
In particular, we find that 
$\mathrm{fib}(F\colon \TR(X) \to \TR(X)^{hC_p})$ is $d$-truncated. 
Taking the cofiber on both the domain and codomain of the Verschiebung, we find
that 
$\mathrm{fib}(F\colon  \TR(X) \to \TR(X)^{hC_p}) = \mathrm{fib}(\varphi\colon  X \to X^{tC_p})$
which is therefore $d$-truncated as desired. 
\end{proof} 

Combining 
\Cref{trunctosegal}, and \Cref{asymptoticK1TR}, we obtain the following result. 
Versions of the Segal conjecture have been studied by many authors. 
For instance, it is known that $\THH(\mathbb{Z}_p; \mathbb{F}_p) \to \THH(\mathbb{Z}_p;
\mathbb{F}_p)^{tC_p}$ is an equivalence on connective covers
\cite[Lem.~6.5]{BM94}. 
Compare \cite{HM03} for the more general case of a DVR $\mathcal{O}_K$ of mixed
characteristic and perfect residue field of characteristic $p>2$. 
The Segal conjecture for smooth algebras in characteristic $p$ appears as 
\cite[Prop.~6.6]{Hes16} and (at the filtered level) \cite[Cor.~8.18]{BMS2}. 
The Segal conjecture has also been verified for certain ring spectra as well,
cf.~\cite{AR02, LNR11, AKQ}.

\begin{corollary}[The Segal conjecture for regular rings] 
Let $R$ be a $p$-torsionfree excellent regular noetherian ring with $R/p$ finitely generated over its
$p$th powers.
Suppose
that
for all $\mathfrak{p} \in \spec(R)$ containing $(p)$, we have
$\dim R_{\mathfrak{p}} + \log_p [ \kappa(\mathfrak{p}):\kappa(\mathfrak{p})^p]
\leq d$. 
Then the map 
$\THH(R; \mathbb{F}_p) \to \THH(R; \mathbb{F}_p)^{tC_p}$ is 
$(d-1)$-truncated.
 \qed
\end{corollary}

\begin{question} 
Is it  possible to prove a filtered version of this
result, with respect to the motivic filtrations \cite{BMS2} on both sides? 
\end{question} 

We next discuss the converse direction. 
Here we only prove the result under a more restrictive hypothesis, namely for
cyclotomic spectra which are $ku$-modules. 

\begin{example} 
\label{ppowerrootsku}
Let $C = \widehat{\overline{\mathbb{Q}_p}}$. 
If $R$ is a $\mathcal{O}_C$-algebra, then the cyclotomic trace
makes $\THH(R; \mathbb{Z}_p)$ into a $ku$-module in view of the equivalence
$ku_{\hat{p}} \simeq K(\mathcal{O}_C; \mathbb{Z}_p)$,
cf.~\cite[Lem.~3.1]{Niz98}, \cite{Suslin}, and \cite{Hes06}. 
One can improve this slightly: $K(
\widehat{\mathbb{Z}_p[\zeta_{p^\infty}]};
\mathbb{Z}_p)$ has the structure of an $E_\infty$-algebra under $ku$, so the
same applies to any 
$\widehat{\mathbb{Z}_p[\zeta_{p^\infty}]}$-algebra  $R$. 
 This
follows from three facts. 
First, 
$K( \widehat{\mathbb{Z}_p[\zeta_{p^\infty}]}; \mathbb{Z}_p) \simeq K(
\widehat{\mathbb{Q}_p(\zeta_{p^\infty})}; \mathbb{Z}_p)$,
cf.~\cite[Lem.~1.3.7]{HN} for this argument (which only uses that the field 
$\widehat{\mathbb{Q}_p(\zeta_{p^\infty})}$ is perfectoid). 
Second, since 
$\widehat{\mathbb{Q}_p(\zeta_{p^\infty})}$  is perfectoid, it has
$p$-cohomological dimension $\leq 1$ by the tilting equivalence, whence 
$K(\widehat{\mathbb{Q}_p(\zeta_{p^\infty})}; \mathbb{Z}_p) \to
L_{K(1)}K(\widehat{\mathbb{Q}_p(\zeta_{p^\infty})}; \mathbb{Z}_p)$  is the
connective cover map, cf.~\Cref{asymptoticKlocalityRO}. Third, 
$L_{K(1)} K(\mathbb{Z}[\zeta_{p^\infty}])$ has the structure of an
$E_\infty$-algebra under $KU$ and hence $ku$, cf.~\cite[Construction 3.7]{BCM}. 
\end{example} 

\begin{lemma} 
\label{S1fplemma}
Let $M \in \fun(BS^1, \mod(H\mathbb{F}_p))$ be an $r$-connected 
object. Suppose that there exists a map $f\colon M \to M'$ in 
$\fun(BS^1, \mod(H\mathbb{F}_p))$ such that $M'$ is induced (as an object with
$S^1$-action) and such that 
$f$ is an equivalence on $\tau_{\geq r}$. 
Then there exists a map $\widetilde{M} \to M$ of $r$-connected objects in
$\fun(BS^1, \mod(H \mathbb{F}_p))$ which induces an equivalence on $\tau_{\geq
r+1}$ and such that $\widetilde{M}$ is induced. 
\end{lemma} 
\begin{proof} 
The homotopy groups $\pi_*(M), \pi_*(M')$ form graded modules over the ring
$\pi_*( \mathbb{F}_p[S^1]) = \mathbb{F}_p[\epsilon]/\epsilon^2, \ |\epsilon| = 1$. 
By assumption, $\pi_*(M')$ is a free graded
$\mathbb{F}_p[\epsilon]/\epsilon^2$-module and $\pi_*(M)$ is the submodule of
those elements in degree $\geq r$. 
Choosing an $\mathbb{F}_p$-subspace $V$ of $\pi_r(M) = \pi_r(M')$ which is complementary to
$\epsilon \pi_{r-1}(M') \subset \pi_r(M')$, we can modify $M$ to form
$\widetilde{M}$ with $\pi_r(\widetilde{M}) = V$; it is now easy to see that
$\pi_*(\widetilde{M})$ is a free graded
$\mathbb{F}_p[\epsilon]/\epsilon^2$-module, so we can conclude.
\end{proof} 

\begin{proposition} 
Let $X$ be a connective, $p$-complete $ku$-module in $\cycsp$. 
Suppose that 
$\varphi\colon  X \to X^{tC_p}$ is $d$-truncated. 
Then the fiber of $\TR(X)/p \to L_{K(1)} \TR(X)/p$ is $ d + 3$-truncated. 
\label{K1locfromcycl}
\end{proposition} 
\begin{proof} 
It follows that the cyclotomic spectrum $Y = X/(p, \beta) = X \otimes_{ku} H
\mathbb{F}_p
\in \cycsp_{\geq 0}$ has the property that
$\varphi\colon  Y \to Y^{tC_p}$ is $d+4$-truncated. 
It follows that the comparison map $Y^{C_p} \to Y^{hC_p}$ is 
$d+4$-truncated as well, via the fiber square
\[ \xymatrix{
Y^{C_p} \ar[d]  \ar[r] & Y^{hC_p} \ar[d]  \\
Y \ar[r]^{\varphi} & Y^{tC_p}
}.\]

Now for each $n \geq 1$, 
we have a pullback diagram
\cite[Lem.~II.4.5]{NS18},
\begin{equation}  \label{ntatesquare} \xymatrix{
Y^{C_{p^n}} \ar[d]^R \ar[r] &  Y^{hC_{p^n}} \ar[d]  \\
Y^{C_{p^{n-1}}} \ar[r] &  Y^{tC_{p^n}}
}. \end{equation}
The bottom horizontal map is $Y^{C_{p^{n-1}}} \to Y^{hC_{p^{n-1}}}
\xrightarrow{\varphi^{hC_{p^{n-1}}}} (Y^{tC_p})^{hC_{p^{n-1}}}$, and the last
term is identified with $Y^{tC_{p^n}}$ using the Tate orbit lemma,
cf.~\cite[Lem.~II.4.1]{NS18}. 

Now $Y^{C_p} \to Y^{hC_p}$ is an equivalence in degrees $\geq d+ 6$, hence 
$Y^{C_{p^{n-1}}} \to Y^{hC_{p^{n-1}}}$ is an equivalence in degrees $\geq d+ 6$
by \cite[Cor.~II.4.9]{NS18} (a generalization of results of
Tsalidis \cite{Tsalidis} and
B\"okstedt--Bruner--Lun\o e-Nielsen--Rognes \cite{BBLR}). 
Since $\varphi \colon  Y \to Y^{tC_p}$ is an equivalence in degrees $\geq d+6$, it
follows that 
the bottom horizontal map in \eqref{ntatesquare} is an equivalence in degrees
$\geq d+ 6$. Note also that $Y^{tC_p}$ is a module over $\mathbb{F}_p^{tC_p}$ 
in $\fun(BS^1, \mod(H\mathbb{F}_p))$; since the latter has induced $S^1$-action,
it follows that the former does too. 

By \Cref{S1fplemma}, we can find a $\geq d+6$-connected $Y'$ with an $S^1$-equivariant map 
$Y' \to 
Y$ which induces an equivalence in degrees 
$\geq d + 7$ and such that $Y'$ is induced as an object of $\fun(BS^1, \mod(H
\mathbb{F}_p))$. 
It follows that the map $Y'^{hC_{p^n}} \to Y^{hC_{p^n}}$ is an equivalence in
degrees $\geq d + 7$. But the commutative square
\[ \xymatrix{
Y'^{hC_{p^n}} \ar[d]  \ar[r] &  Y^{hC_{p^n}} \ar[d]  \\
Y'^{tC_{p^n}} = 0 \ar[r] &  Y^{tC_{p^n}}
}\]
now shows that any $\alpha \in \pi_r( Y^{hC_{p^n}})$ for $r \geq d + 7$ has
vanishing image in $\pi_r (Y^{tC_{p^n}})$. 
Using the commutative square 
\eqref{ntatesquare} again (since the horizontal maps are equivalences in
degrees $\geq d+6$), it follows that the restriction map 
$Y^{C_{p^n}} \to Y^{C_{p^{n-1}}}$ induces the zero map in degrees $\geq d + 7$. 
Consequently, taking the inverse limit over $R$, we find that $\TR(Y) \in
\sp_{\leq d + 6}$, whence the fiber of $\TR(X)/p \to L_{K(1)} \TR(X)/p$ belongs to $\sp_{\leq d + 3}$
by \Cref{K1locallemma}. 
\end{proof} 
\begin{lemma} 
\label{K1locallemma}
Let $M$ be a $ku$-module spectrum. 
Then the following are equivalent: 
\begin{enumerate}
\item  
$M/(p, \beta)$ is concentrated in degrees $\leq d + 3$. 

\item

The fiber of $M/p \to L_{K(1)} M/p$ is concentrated in degrees $\leq d$. 
\end{enumerate}
\end{lemma} 
\begin{proof} 
Given a $ku$-module $N$, it suffices to show that the fiber of $N \to 
N[1/\beta]$ (which after $p$-completion is $L_{K(1)} N$) is
concentrated in degrees $\leq d$ if and only if $N/\beta$ is concentrated in
degrees $\leq d+3$; then the result will follow by taking $N = M/p$.
To prove the claim, first 
replace $N$ by $\mathrm{fib}( N \to N[1/\beta])$; thus we may assume that $N$ is
actually $\beta$-power torsion. 

Then, the claim follows from the observation that a $\beta$-power torsion
$ku$-module $N$
is concentrated in degrees $\leq d$ if and only if $N/\beta$ is concentrated in
degrees $\leq d+3$. The ``only if'' direction is evident, so we verify the
``if'' direction. 
In fact, if there exists a nonzero $x \in \pi_i(N)$ for $i > d$, then 
by multiplying by a power of $\beta$ we may assume $\beta x = 0$, whence
there exists a nonzero class in $\pi_{i+3}(N/\beta)$ from the long exact
sequence. 
\end{proof}

\begin{remark} 
Let $X = \THH(\mathbb{F}_p)$; in this case, we have that $\THH(\mathbb{F}_p) \to
\THH(\mathbb{F}_p)^{tC_p}$ is $(-3)$-truncated. 
This computation is due to Hesselholt--Madsen \cite[Sec.~5]{HM97}, and is
refined in \cite[App.~IV-4]{NS18}. 
Meanwhile $\TR(\mathbb{F}_p) =
H\mathbb{Z}_p\to
L_{K(1)} \TR(\mathbb{F}_p)
= 0$ is $0$-truncated modulo $p$. 
Thus, the bound of \Cref{K1locfromcycl} is the best possible. 
\end{remark}

\begin{proposition} 
\label{perfectoidtrunc}
Let $R_0$ be a $p$-torsionfree perfectoid ring. 
Let $R$ be a formally smooth $R_0$-algebra
(with respect to the $p$-adic topology) of relative dimension $d$. 
Then 
the map $\varphi \colon \THH(R; \mathbb{F}_p) \to \THH(R; \mathbb{F}_p)^{tC_p}$
is $(d-3)$-truncated, and  
the map $\TR(R; \mathbb{F}_p) \to L_{K(1)} \TR(R; \mathbb{F}_p)$
is $(d-1)$-truncated. 
\end{proposition} 

\begin{proof} 
Let $X_i, i \in \mathcal{C}$ be a diagram of spectra. 
Suppose that for each $i \in \mathcal{C}$, the map $X_i \to L_{K(1)} X_i$ is
$m$-truncated for some $m$. Then the map $\varprojlim_i X_i \to L_{K(1)} ( \varprojlim_i
X_i)$ is $m$-truncated, and $L_{K(1)} ( \varprojlim_i X_i) \to \varprojlim_i
L_{K(1)} X_i$ is an equivalence. This follows because $L_{K(1)}(-)$ annihilates
bounded-above spectra. 
Applying the above observation, 
and using flat descent of $\TR(-; \mathbb{Z}_p)$ (which follows from
\cite[Sec.~3]{BMS2}) and $\THH(-; \mathbb{Z}_p), \THH(-; \mathbb{Z}_p)^{tC_p}$,
we 
reduce to the case where $R$ is a $\widehat{\mathbb{Z}_p[\zeta_{p^\infty}]}$-algebra (e.g., using Andr\'e's lemma,
cf.~\cite[Th.~7.12]{Prisms}), so 
$\THH(R; \mathbb{Z}_p)$ is a $ku_{\hat{p}}  $-module in
$\cycsp_{\geq 0}$ (\Cref{ppowerrootsku}). 
With this reduction in mind, the first claim implies the second 
in view of \Cref{K1locfromcycl} (applied to the cyclotomic spectrum $\THH(R;
\mathbb{F}_p)$). 

Thus, we prove the first claim (i.e., the Segal conjecture for $\THH(R;
\mathbb{F}_p)$); this result appears as \cite[Cor.~9.12]{BMS2} in
the case where $R_0  $ is the ring of integers in a complete, algebraically
closed nonarchimedean field of mixed characteristic, and in \cite[Sec.~6]{BMS2}
when $R  = R_0$. 

Let $(A, I)$ be the perfect prism associated to 
the perfectoid ring $R_0$, and let $\widetilde{\xi} \in I$ be a generator. 
We use the prismatic cohomology $\Prism_{R/A}$ of \cite{Prisms} and its Nygaard completion $\Prismh_{R/A}$. 
For a quasisyntomic $R_0$-algebra $S$, 
there are \cite{BMS2} complete, exhaustive descending $\mathbb{Z}$-indexed
filtrations on $\TC^-(S; \mathbb{Z}_p), \TP(S; \mathbb{Z}_p)$
with 
$$\gr^i \TC^-( S; \mathbb{Z}_p) \simeq \mathcal{N}^{\geq i} \Prismh_{S/A}[2i], 
\quad \gr^i \TP(S; \mathbb{Z}_p) \simeq \Prismh_{S/A}[2i],
$$ 
and the map $\varphi \colon \TC^-(S; \mathbb{Z}_p) \to \TP(S;
\mathbb{Z}_p)$ 
on graded pieces is given by 
the prismatic divided Frobenius \begin{equation} \label{varphigr}
\varphi/\widetilde{\xi}^i
\colon \mathcal{N}^{\geq i} \Prismh_{S/A}[2i] \to
\Prismh_{S/A}[2i],  \end{equation}
cf.~\cite[Sec.~13]{Prisms} for the comparison between prismatic cohomology and
the objects of \cite{BMS2}. 
Here we have trivialized the Breuil--Kisin twists involved since we are over the
base perfectoid ring $R_0$, and we can compute the Nygaard-completed absolute
prismatic cohomology as the Nygaard completed relative prismatic cohomology
over $R_0$.

Now the map $\varphi \colon \TC^-(S; \mathbb{Z}_p) \to \TP(S;
\mathbb{Z}_p)$ arises
by taking $S^1$-invariants on the cyclotomic Frobenius $\varphi \colon \THH(S; \mathbb{Z}_p) \to
\THH(S; \mathbb{Z}_p)^{tC_p}$. Both sides of this map are filtered (again, as
in \cite{BMS2}) and the
associated graded pieces are 
\[ \gr^i \THH(S; \mathbb{Z}_p) \simeq \mathcal{N}^{i} \Prismh_{S/A}[2i], \quad \gr^i
\THH(S; \mathbb{Z}_p)^{tC_p} = (\Prismh_{S/A})/\widetilde{\xi} [2i],\]
and the cyclotomic Frobenius
is induced from \eqref{varphigr}. 
But since $R$ is formally smooth over $R_0$, the Nygaard filtration is complete  and 
$\varphi/\widetilde{\xi}^i$ induces an equivalence
$\mathcal{N}^i \Prism_{R/A} \xrightarrow{\sim} \tau^{\leq i} ( \Prism_{R/A}/
\widetilde{\xi})$, cf.~\cite[Sec.~12.4]{Prisms}, where the right-hand-side is the $i$th stage
of the conjugate filtration on $\Prism_{R/A}/\widetilde{\xi}$. 
Using the Hodge--Tate comparison, \cite[Th.~4.10]{Prisms}, it follows that 
the cohomology groups of $\Prism_{R/A}/\widetilde{\xi}$ are $p$-torsion-free. 
It follows that 
$$ \gr^i \varphi \colon  \gr^i \THH(R; \mathbb{F}_p) \to  \gr^i \THH(R;
\mathbb{F}_p)^{tC_p}$$
exhibits the domain as the $i$-connective cover of the codomain, and therefore
has $(i -2)$-truncated homotopy fiber for all $i$. The codomain is
$(2i-d)$-connective 
by the Hodge--Tate comparison, since $R/R_0$ is formally smooth of
relative dimension $d$; therefore, $\gr^i \varphi$ is an equivalence for $i \geq
d$.
Therefore, the fiber of $\varphi \colon \THH(R; \mathbb{F}_p) \to \THH(R;
\mathbb{F}_p)^{tC_p}$  
has a complete filtration such that $\gr^i $ is $(i-2)$-truncated for all $i$
and contractible for $i \geq d$. This proves that 
$\varphi \colon \THH(R; \mathbb{F}_p) \to \THH(R;
\mathbb{F}_p)^{tC_p}$
is $(d-3)$-truncated, whence the result. 
\end{proof}

\section{Pro-Galois descent}

In this section, we prove a type of pro-Galois descent for $\TR$ in the generic
fiber, which is related to the conjecture in \cite{HesICM}. The basic example is when $K$ is a characteristic zero local field, 
and one tries to relate $\TR(\mathcal{O}_K; \mathbb{Z}_p)$ (computed by \cite{HM03}) with the ``continuous'' homotopy fixed points
for the Galois group $\mathrm{Gal}(\overline{K}/K)$ on $\TR(
\mathcal{O}_{\overline{K}}; \mathbb{Z}_p)$, before or after $K(1)$-localization. 
The advantage is that the latter is much more tractable, 
cf.~\cite{Hes06} for the calculation of $\TR( \mathcal{O}_{\overline{K}};
\mathbb{Z}_p)$. 
For finite Galois extensions in the generic fiber, these claims follow from
section~\ref{generalK1loctrunc}. However, there are some additional subtleties
to extend to pro-Galois descent because $\TR$ fails to commute with filtered
colimits.

\subsection{An auxiliary construction}
Let $B$ be a base ring. 
Let $E$ be a $K(1)$-local, localizing invariant on $B$-linear
$\infty$-categories which is truncating. 
Let $R_0$ be a $p$-adically complete $B$-algebra. 
Given a $K(1)$-local truncating invariant $E$, 
we now describe a construction of a sheaf on the finite \'etale site of
$R_0[1/p]$. 
For every finite \'etale $R_0[1/p]$-algebra $S$, we can choose a ``ring of
integers'' $S_0 $ which is finite and finitely presented over $R_0$ with $S_0[1/p] =  S$  and consider $E(S_0)$. 
It is not difficult to see that this only depends on $S$ and that it defines the
desired sheaf; to make the functoriality precise, we use left Kan extension.

\begin{construction}[$ E$ as a sheaf on the finite \'etale site of the generic fiber]
Let $R_0$ be a $p$-adically complete $B$-algebra and let $R = R_0[1/p]$. 
Using the $K(1)$-local, localizing invariant $E$ which is assumed to
be truncating, we define a sheaf $\sF_E$ of spectra on the finite \'etale 
site of $\spec(R)$ as follows. 
Let
$\mathcal{C}_0$ denote the category of finite, finitely presented
$R_0$-algebras $S_0$ with $S_0[1/p]$ \'etale over $R$, and let $\mathcal{C}$
denote the category of finite \'etale $R$-algebras. 
Consider the functor $F\colon  \mathcal{C}_0 \to \mathcal{C}$ given by inverting $p$.
Note that: 
\begin{enumerate}
\item $ F$ is essentially surjective. 
That is, given any finite \'etale $R$-algebra $S$, there exists a finite,
finitely presented $R_0$-algebra $S_0$ with $S = S_0[1/p]$. 
In fact, consider 
any finite $R_0$-subalgebra $S_0' \subset S$ with $S_0'[1/p] = S$.
The algebra $S_0'$ is not necessarily finitely presented, but it is a directed
colimit of finite, finitely presented $R_0$-algebras $S_0^{(\alpha)}$ under surjective maps; one of
them will have $S_0^{(\alpha)}[1/p] = S$, and can be taken for $S_0$. 
\item If $S \in \mathcal{C}$, then $\mathcal{C}_0
\times_{\mathcal{C}}\mathcal{C}_{/S}$ is a filtered category. 
In fact, 
the subcategory of 
$\mathcal{C}_0
\times_{\mathcal{C}}\mathcal{C}_{/S}$ spanned by those $S_0$ such that the
structure map $S_0[1/p] \to S$ is an isomorphism is itself filtered and cofinal. 
This follows similarly by comparing $S_0$ with its image in $S$.
\end{enumerate}

We consider the functor 
$S_0 \mapsto E(S_0)$ on the category $\mathcal{C}_0$ of finite, finitely presented
$R_0$-algebras $S_0$ with $S_0[1/p]$ finite \'etale over $R$. 
Then, to define $\sF_E$ on 
finite \'etale $R$-algebras, 
we left Kan extend $E(-)$ along the functor $\mathcal{C}_0 \to \mathcal{C}$. 
Explicitly, it follows 
from \Cref{rigidgenericfiberproper}
that if $S$ is a finite \'etale $R$-algebra and 
$S_0$ is a finite, finitely presented $R_0$-algebra with $S_0[1/p] \simeq S$,
then 
we have a canonical equivalence
$\sF_E(S) \simeq E(S_0)$. 
It also follows from $h$-descent that 
$\sF_E$ is indeed a sheaf of $K(1)$-local spectra on the finite \'etale site of
$\spec(R)$. 
Note that it is a sheaf of modules over the sheaf $S \mapsto L_{K(1)} K(S)$,
which is also the sheaf $\sF_{L_{K(1)} K}$  because $L_{K(1)} K(-)$ is insensitive
to inverting $p$ on $H\mathbb{Z}$-algebras as in 
\eqref{k1invertp}. 

\end{construction}

\begin{proposition}[Hypercompleteness of $\sF_E$] 

\label{FEhypcomplete}
Notation as above, suppose there is a uniform bound on the mod $p$ cohomological dimensions of
the residue fields of $R$ and $R$ has finite Krull dimension. 
Then $\sF_E$ defines a hypercomplete sheaf on the finite \'etale site of
$\spec(R)$. 
\end{proposition}

\begin{proof} 
Our hypotheses imply that $\spec(R)$ has finite mod $p$ \'etale
cohomological dimension, cf.~\cite[Cor.~3.29]{CM19}. 

Now we use 
the $K(\pi, 1)$ property, which states that for any $p$-henselian ring $A$ and
any $\mathbb{F}_p$-local system $\mathcal{L}$ on $\spec(A[1/p])$, the 
cohomology of $\mathcal{L}$ on the \'etale sites and the finite \'etale sites of
$\spec(A[1/p])$ coincide. 
For $A$ noetherian and $p$-complete, this is proved in \cite[Th.~4.9]{Sch13}. 
The case of $A$ noetherian and $p$-henselian then follows 
 using the Fujiwara--Gabber theorem
(cf.~\cite[Th.~6.11]{BM18} for an account), and then one can pass to the limit
to obtain the arbitrary $p$-henselian case using
 the commutation of cohomology (on either the \'etale or finite \'etale
 sites) and filtered colimits.

From the above two paragraphs, 
it follows that the \'etale fundamental group $\pi_1^{\mathrm{et}}( \spec( R))$ has
finite mod $p$ cohomological dimension, which implies that the notion of
hypercompleteness for sheaves of $p$-complete spectra on the finite \'etale site
of $\spec(R)$ can be made explicit in terms of exponents of nilpotence
\cite[Sec.~4.1]{CM19}. 
Now
$\sF_E$ is a module over
the hypercomplete sheaf $S \mapsto L_{K(1)}K(S)$ (cf.~\cite[Th.~7.14]{CM19} for
hypercompleteness),
hence a hypercomplete sheaf itself, thanks to \cite[Cor.~4.26]{CM19}. 
\end{proof}

We also need the following variant of the above with respect to a fixed
profinite group. 
\begin{construction}[$\sF_E$ relative to a profinite group] 
\label{relativeFE}
Let $S_0$ be an $R_0$-algebra equipped with the action of a profinite group $\Gamma$
which is continuous (with respect to the discrete topology on $S_0$), and such
that $R \to S := S_0[1/p]$ is $\Gamma$-Galois. 
Suppose that there exist a cofinal collection of open normal subgroups
$N_i \subset \Gamma, i \in I$ for which the fixed points $S_0^{N_i}$ form a
finite, finitely presented $R_0$-algebra. 
It follows from the above that we obtain a sheaf on the category of finite
continuous $\Gamma$-sets 
which carries $\Gamma/N_i \mapsto E( S_0^{N_i})$. 
This is just the restriction of $\sF_E$ to the site of finite continuous
$\Gamma$-sets (which maps to the finite \'etale site of $R$). 
\end{construction} 

When is the sheaf of \Cref{relativeFE} hypercomplete? 
When $\Gamma$ has finite cohomological dimension, 
then 
hypercompleteness is smashing \cite[Sec.~4.1]{CM19}, so hypercompleteness holds if 
the sheaf 
of spectra which sends
$\Gamma/H  \mapsto L_{K(1)} K( S^H)$ (for any cofinal collection of open
normal subgroups $H$) is
hypercomplete.

\begin{example} 
\label{ppowerrootex}
Suppose $\Gamma = \mathbb{Z}_p^n$ and the $\Gamma$-extension of $R_0$ is obtained
by adding compatible systems of $p$-power roots. 
Explicitly, suppose $R_0$ is a $\mathbb{Z}[ \zeta_{p^\infty}, t_1^{\pm 1}, \dots, t_n^{\pm
1}]$-algebra and $S_0 =R_0 \otimes_{\mathbb{Z}[ \zeta_{p^\infty}, t_1^{\pm 1}, \dots, t_n^{\pm
1}]}
\mathbb{Z}[ \zeta_{p^\infty}, t_1^{\pm 1/p^\infty}, \dots, t_n^{\pm
1/p^\infty}]
,$ with the evident $\Gamma$-action.  
In this case, the 
sheaf of \Cref{relativeFE} is hypercomplete. 
Again since hypercompleteness is smashing, this follows because
any $K(1)$-local localizing invariant 
yields a hypercomplete \'etale sheaf on the site 
of finite \'etale (or even all \'etale) $\mathbb{Z}[1/p, \zeta_{p^\infty},  t_1^{\pm 1}, \dots,
t_n^{\pm 1}]$-algebras, cf.~\cite[Th.~7.14]{CM19}. 
In particular, we take as the localizing invariant
$A \mapsto L_{K(1)} K(A \otimes_{\mathbb{Z}[1/p, \zeta_{p^\infty}, t_1^{\pm 1},
\dots, t_n^{\pm 1}]}
R)$. 
\end{example}

\subsection{Completion of topological Cartier modules}
Here again we use the ``decompletion'' of the theory of
cyclotomic spectra given by the topological Cartier modules of
Antieau--Nikolaus \cite{AN} and an
amplitude property of the completion. 
We recall that $\TR(-)$ gives a fully faithful right adjoint embedding from
bounded-below cyclotomic spectra into bounded-below topological Cartier modules, with 
image the $V$-complete objects, cf.~\cite[Th.~3.21]{AN}. 

\begin{proposition} 
\label{completionCart}
Let $M$ be a   topological Cartier module which is $d$-truncated and bounded
below. 
Then the $V$-completion of $M$ is $d+3$-truncated. 
If $M$ is $p$-complete, then the $V$-completion of $M$ is $d+2$-truncated. 
\end{proposition} 

\begin{proof} 
We reduce by d\'evissage to the case where $M$ is  concentrated in degree zero, 
and $d = 0$. 
The completion of $M$  is given by 
$\mathrm{cofib}( \varprojlim M_{hC_{p^n}} \to M)$ where the maps 
$M_{hC_{p^n}} \to M_{hC_{p^{n-1}}}$ are given by $V_{hC_{p^{n-1}}}$,
cf.~\cite[Prop.~3.22]{AN}. 

Now since $M$ is discrete, the Verschiebung is simply given by
a map $V\colon  M \to M$ of abelian groups, and the $S^1$-action is trivial. 
We can (as in the proof of \cite[Lem.~3.25]{AN}) 
form a $\mathbb{Z}_{\geq 0} \times \mathbb{Z}_{\geq 0}$-indexed
diagram
$Y_{i, j} = M \otimes BC_{p^j + }$ such that the transition maps in the 
$i$-direction are given by $V$ and in the $j$-direction are given by the
canonical projections. 
By the above, the completion of $M$ is given by the cofiber of the
map $\varprojlim_{i, j} Y_{i, j} \to M$. 

Now a simple computation shows that for any abelian group $A$, $\varprojlim_{j}
HA\otimes
BC_{p^j + }$ are concentrated in degrees $\leq 2$, and degrees $\leq  1$ if $A$
is derived $p$-complete. This claim will imply the result.  
Indeed, for the first part, it suffices to show that 
$\varprojlim_{j}
HA\otimes
BC_{p^j + }$
is concentrated in degrees $\leq 1$ when $A$ is torsion-free; this 
follows  because the pro abelian group $\{H_*(BC_{p^j};
\mathbb{Z}) \}_{j \geq 0} $ is pro-zero for $\ast \geq 2$. 
For the second part, we also observe that the pro-abelian group
$\left\{H_1( BC_{p^j}; \mathbb{Z})\right\}_{j \geq 0}$ is simply the tower
$\dots \to \mathbb{Z}/p^2  \to \mathbb{Z}/p \to 0$, 
and 
our assumption that $A$ is derived $p$-complete gives
$A \simeq \varprojlim A \otimes^{\mathbb{L}}_{\mathbb{Z}} \mathbb{Z}/p^j$
is in particular discrete. 
\end{proof}

\begin{remark} 
We can give another proof of \Cref{completionCart} using 
the results from the previous section in the case where $M$ is annihilated by a
power of $p$. By d\'evissage, we can reduce to the case where $M$ is an
$H\mathbb{F}_p$-module (in topological Cartier modules). 
Let $X = M/M_{hC_p} = \mathrm{cofib}(V)$ be the associated cyclotomic spectrum, so $\TR(X)$ is the derived
$V$-completion of $M$ by the correspondence between bounded-below $V$-complete
Cartier modules and bounded-below cyclotomic spectra, cf.~\cite[Th.~3.21]{AN}. 
Then the proof of \Cref{trunctosegal} shows that 
the cyclotomic Frobenius
$\varphi\colon  X \to X^{tC_p}$ is $d$-truncated, since $\mathrm{fib}(\varphi) =
\mathrm{fib}(F\colon  M \to M^{hC_p})$. 
Using \Cref{K1locfromcycl}, 
we find that $\TR(X)/p$ is $d+3$-truncated, whence $\TR(X)$ is $d+2$-truncated. 

\end{remark}

\begin{corollary} 
\label{colimtruncatedTR}
Let $R_i, i \in I$ be a filtered system of rings and let $R = \varinjlim R_i$. 
Suppose that 
the map $\TR(R_i; \mathbb{F}_p) \to L_{K(1)} \TR(R_i ; \mathbb{F}_p)$ is
$d$-truncated for all $i \in I$. 
Then $\TR(R; \mathbb{F}_p) \to L_{K(1)} \TR(R; \mathbb{F}_p)$ is $
d+2$-truncated. 
\end{corollary} 
\begin{proof} 
Let $v\colon  \Sigma^u (S^0/p) \to (S^0/p)$ be a $v_1$-self map (so we can take $u  =
2p-2$ for $p$ odd and $u = 8 $ for $p = 2$). 
Let $X$ be any spectrum. Then we 
observe that the following are equivalent: 
\begin{enumerate}
\item The map $X/p \to L_{K(1)} X/p$ is $d$-truncated.  
\item The spectrum $X \otimes S^0/(p, v)$ is $d + u+ 1$-truncated. 
\end{enumerate}
The equivalence is proved analogously to \Cref{K1locallemma}, noting that $L_{K(1)} (X/p) =
X \otimes (S^0/p)[v^{-1}]$ by the telescope conjecture at height one
\cite{Ma81, Mi81}. 

Therefore, it suffices to show that 
$\TR(R; \mathbb{Z}_p) \otimes S^0/(p, v)$ is $d + u+3$-truncated. 
To this end, let $M = \varinjlim \TR(R_i; \mathbb{Z}_p)$, so $M$ is a
connective topological Cartier module which is not necessarily derived
$V$-complete; the 
$V$-completion
of its $p$-completion
is $\TR(R; \mathbb{Z}_p)$ since $\THH(-)$ commutes with filtered colimits (as a
functor from rings to $\cycsp_{\geq 0}$). 
Now the topological Cartier module $M \otimes S^0/(p, v)$ is $d + u +
1$-truncated as a filtered colimit of $d + u + 1$-truncated objects. 
Taking the $V$-completion and using 
\Cref{completionCart}, 
we find that $\TR(R; \mathbb{Z}_p)/(p, v)$ is $d + u + 3$-truncated. 
Therefore, $\TR(R; \mathbb{F}_p) \to L_{K(1)}
\TR(R; \mathbb{F}_p)$ is $d+2$-truncated. 
\end{proof}

\subsection{The main pro-Galois result}
In this subsection, we prove the following pro-Galois descent result. 

\begin{theorem}[Pro-Galois descent in the generic fiber] 
\label{progaldesc}
Fix $d \geq 0$. 
Let $R$ be a $p$-complete ring such that $\TR(R; \mathbb{F}_p) \to L_{K(1)} \TR(R;
\mathbb{F}_p)$ is $d$-truncated. 
Let $S$ be an $R$-algebra.

Let $G$ be a profinite group of finite $p$-cohomological dimension which acts
continuously on
the $R$-algebra $S$  (given the discrete topology). 
Suppose that: 
\begin{enumerate}
\item $R[1/p] \to S[1/p]$ is a $G$-Galois extension. 
\item There is a cofinal collection of open normal subgroups $N_i \subset G, i
\in I$ such that $S_i := S^{N_i}$ is 
a finite, finitely presented $R$-algebra
and such that 
$\TR(S_i; \mathbb{F}_p) \to L_{K(1)} \TR(S_i; \mathbb{F}_p)$ is $d$-truncated. 
\item 
The induced 
sheaf of spectra
on 
finite continuous $G$-sets given by $T \mapsto L_{K(1)} ( K( \mathrm{Fun}_G(T,
S[1/p])))$ 
(for $\mathrm{Fun}_G(T, S[1/p])$ denoting $G$-equivariant functions $T \to S$)
is hypercomplete. 
For example, this holds if $R[1/p]$ has finite Krull dimension
and there is a uniform bound on the mod $p$ cohomological dimensions of the
residue fields, but also in cases such as \Cref{ppowerrootex}. 
\end{enumerate}
Then
the map 
\begin{equation} L_{K(1)}\TR(R) \to (L_{K(1)}\TR(S))^{h
G_{\mathrm{cts}}} := \mathrm{Tot}( 
L_{K(1)} \TR(S) \rightrightarrows L_{K(1)} \TR( \mathrm{Fun}_{\mathrm{cts}}(G,
S)) \triplearrows \dots ). \label{RSequiv} \end{equation}
is an equivalence
and the map 
$$ \TR(R; \mathbb{F}_p) \to \mathrm{Tot}( \TR(S; \mathbb{F}_p) \rightrightarrows \TR(
\mathrm{Fun}_{\mathrm{cts}}(G, S); \mathbb{F}_p) \triplearrows \dots ) $$
is $d+2$-truncated. 

\end{theorem} 
\begin{proof} 

Let $\sF$ be any product-preserving presheaf from finite continuous $G$-sets 
to $p$-torsion spectra. 
We let $\sF^{\mathrm{disc}}$ denote the left Kan extension to 
profinite $G$-sets, so 
if $S$ is a profinite $G$-set which can be written $S \simeq \varprojlim S_j$
for some finite continuous $G$-sets $S_j$, 
then
$\sF^{\mathrm{disc}}(S) \simeq 
\varinjlim \sF(S_j)$. 
We let 
\[ R \Gamma( \ast, \sF) = \mathrm{Tot}( 
\sF^{\mathrm{disc}}(G) \rightrightarrows \sF^{\mathrm{disc}}(G \times G) \triplearrows
\dots ) . \]
As in \cite[Sec.~4.1]{CM19}, $R \Gamma( \ast, \sF)$ is the 
value of the hypersheafification or Postnikov sheafification of $\sF$ (with respect to the topology on
finite continuous $G$-sets where covering families are jointly
surjective ones) at $\ast$. 
When we work with $p$-torsion spectra, our assumption that $G$ 
has finite cohomological dimension implies that $R \Gamma(\ast, -)$ commutes
with all colimits. 

Now we take $\sF$ to  be the presheaf 
which sends 
a finite continuous $G$-set $T$ to $\TR( \mathrm{Fun}_G(T, S); \mathbb{F}_p)$, where
$\mathrm{Fun}_G(T, S)$ is (as in the statement) the ring of $G$-equivariant functions $T \to S$.  Unwinding the definitions and hypotheses, we find 
from \Cref{galoisdescgeneric}
that if $T = G/N_j$ for one of
the distinguished normal subgroups $N_j$, we have that
\[ L_{K(1)}\sF(\ast)  \xrightarrow{\sim} \mathrm{Tot}( L_{K(1)}\sF(T) \rightrightarrows
L_{K(1)}\sF(T \times T)
\triplearrows \dots )  \]
is an equivalence. 
Note $L_{K(1)} \sF$ is an example of 
\Cref{relativeFE} for the profinite group $G$ (though we denote by $R$ the
$p$-complete ring). 
In particular, our assumptions imply hypercompleteness of $L_{K(1)} \sF$, 
so passing to the limit
we find 
\begin{equation} \label{LK1Fast} L_{K(1)} \sF(\ast)  
\simeq R \Gamma( \ast, L_{K(1)}\sF)
. \end{equation}
Note here that $\sF^{\mathrm{disc}}$ is not given by $\TR(-; \mathbb{F}_p)$ because $\TR$ does
not commute with filtered colimits; 
instead, for instance, $\sF^{\mathrm{disc}}( G) = \varinjlim_{i \in I} \TR(S_i;
\mathbb{F}_p)$. 
By assumption (2), $\sF^{\mathrm{disc}}(G), 
\sF^{\mathrm{disc}}(G \times G) , \dots $ are spectra with the property that the map to
their $K(1)$-localization is $d$-truncated; therefore, 
$R \Gamma(\ast, \sF) \to L_{K(1)}R \Gamma(\ast, \sF)=  R \Gamma(\ast, L_{K(1)} \sF)$ is
$d$-truncated.

Now we 
apply 2-out-of-3 to the sequence of maps $\sF(\ast) \to R \Gamma(\ast, \sF)
\to R \Gamma(\ast, L_{K(1)} \sF)$.
We just showed that the second map is $d$-truncated, while the composite map is 
by \eqref{LK1Fast} identified with $\TR(R; \mathbb{F}_p) \to L_{K(1)} \TR(R;
\mathbb{F}_p)$, which is $d$-truncated by assumption. 
Therefore, by 2-out-of-3, 
we find that $\sF(\ast) \to R \Gamma(\ast, \sF)$
is $d$-truncated. 
In other words,
\begin{equation} \label{TruncTot} \sF(\ast) \to \mathrm{Tot}( \sF^{\mathrm{disc}}(G) \rightrightarrows
\sF^{\mathrm{disc}}(G \times G)
\triplearrows \dots ) , \end{equation}
is $d$-truncated. 
Here both sides of \eqref{TruncTot} have the structure of topological Cartier
modules since the
forgetful functor from topological Cartier modules to spectra commutes with
limits and colimits \cite[Prop.~3.11]{AN}. 

Now we take the $V$-completion of both sides in \eqref{TruncTot} (considered as
topological Cartier modules). 
The left-hand-side 
of \eqref{TruncTot}
is already $V$-complete. 
To analyze the right-hand-side, observe that 
the totalization in \eqref{TruncTot} commutes with $(-)_{hC_{p^{n-1}}}$ because
$R\Gamma( \ast, -)$ commutes with colimits. 
The limit over $n$ in computing the $V$-completion clearly commutes with
the totalization. Thus, the $V$-completion of the right-hand-side of
\eqref{TruncTot} is given by 
$$\mathrm{Tot}\left( \TR(S; \mathbb{F}_p) \rightrightarrows \TR(
\mathrm{Fun}_{\mathrm{cts}}(G, S); \mathbb{F}_p) \triplearrows \dots \right).$$

Note finally that the right-hand-side of \eqref{TruncTot} is bounded-below by
the finiteness of the $p$-cohomological dimension. 
Taking the $V$-completion in \eqref{TruncTot}, we find from \Cref{completionCart} that 
\[ \TR(R; \mathbb{F}_p) \to \mathrm{Tot}\left( \TR(S; \mathbb{F}_p) \rightrightarrows \TR(
\mathrm{Fun}_{\mathrm{cts}}(G, S); \mathbb{F}_p) \triplearrows \dots \right)  \]
is $(d+2)$-truncated, whence the last claim of the theorem. 
Finally, $\TR(S; \mathbb{F}_p), \TR( \mathrm{Fun}_{\mathrm{cts}}(G, S);
\mathbb{F}_p)$ map via $(d+2)$-truncated maps to their $K(1)$-localizations,
via \Cref{colimtruncatedTR}, so the $K(1)$-local descent statement follows. 
\end{proof}

\begin{example}[Discrete valuation rings] 
\label{DVRex}
Let $K$ be a complete, discretely valued field of characteristic zero whose residue field $k$
is of characteristic $p$ with $[k: k^p] \leq p^d$. 
It follows that 
if  $k'/k$ is any finite extension, then $[k':k'^p] \leq p^d$,
cf.~\cite[Lem.~2.1.1]{GO08}.

By the main result of \cite{GO08} (which is due to 
\cite{Kat82} in this case), it follows that the Galois group
$\mathrm{Gal}(\overline{K}/K)$ has $p$-cohomological dimension $\leq d + 2$. 
Moreover, by \Cref{asymptoticK1TR}, it follows that 
if $L$ is any finite extension of $K$ and $\mathcal{O}_L \subset L$ the ring of
integers (which is excellent as a complete local ring), then 
$\TR( \mathcal{O}_L; \mathbb{F}_p) \to L_{K(1)} \TR(\mathcal{O}_L;
\mathbb{F}_p)$ is 
$d$-truncated.

We can now apply \Cref{progaldesc} to conclude that
\[ \TR( \mathcal{O}_K; \mathbb{F}_p) \to
\TR(\mathcal{O}_{\overline{K}}; \mathbb{F}_p)^{h
\mathrm{Gal}_K^{\mathrm{cts}}} := \mathrm{Tot}(
\TR(\mathcal{O}_{\overline{K}}; \mathbb{F}_p) \rightrightarrows
\TR(\mathcal{O}_{\overline{K} \otimes_K \overline{K}}; \mathbb{F}_p)
\triplearrows \dots )
\]
is $d$-truncated. 
Note that the theorem gives a priori that the map is $(d+2)$-truncated, but we can
upgrade the conclusion to $d$-truncated as follows: 
the analogous comparison map with $L_{K(1)} \TR(-)$ everywhere is an equivalence; 
the maps  $\TR( \mathcal{O}_{\overline{K} \otimes_K \dots
\otimes_K\overline{K}}; \mathbb{F}_p) \to \TR( \mathcal{O}_{\overline{K} \otimes_K \dots
\otimes_K\overline{K}}; \mathbb{F}_p)$ are $(-1)$-truncated 
by \Cref{perfectoidtrunc} since these are $p$-torsionfree rings whose
completions are perfectoid; 
and the map $\TR(\mathcal{O}_K; \mathbb{F}_p) \to L_{K(1)} \TR(\mathcal{O}_K;
\mathbb{F}_p)$ is $d$-truncated (\Cref{asymptoticK1TR}).

Suppose in particular that $k$ is perfect and of characteristic
$>2$.\footnote{We expect that this works when $k$ has characteristic $2$ as
well.} 
In this case, the results of \cite{HM03} (recalled in 
the proof of \Cref{HMasymptotic}) 
show that 
$\TR(\mathcal{O}_K|K; \mathbb{F}_p)$ is the connective cover of its
$K(1)$-localization, which is $L_{K(1)} \TR(\mathcal{O}_K; \mathbb{F}_p)$. 
Again, $\mathcal{O}_{\overline{K}}, \mathcal{O}_{\overline{K} \otimes_K
\overline{K}}, \dots$ have $p$-completions which are perfectoid rings, so for them 
$\TR(-; \mathbb{F}_p)$ agrees with the connective cover of its
$K(1)$-localization (\Cref{perfectoidtrunc}). 
It follows that
\( \TR(\mathcal{O}_K|K; \mathbb{F}_p) \simeq \tau_{\geq 0}
\TR(\mathcal{O}_{\overline{K}} ; \mathbb{F}_p)^{h \mathrm{Gal}_K^{\mathrm{cts}}}
.\)
\end{example} 

\begin{example} 
\label{smoothoverOC}
Let $R_0$ be a 
$p$-torsionfree perfectoid ring containing a system of primitive $p$-power roots
of unity. 
Let $R$ be a formally smooth $R_0$-algebra, which is formally \'etale over the
formal torus  (i.e., $p$-completed Laurent polynomial algebra) $R_0 \left \langle  t_1^{\pm 1}, \dots,
t_n^{\pm 1} \right\rangle$. 

Let $ G =\mathbb{Z}_p(1)^n$ and let 
$S = R \otimes_{R_0\left \langle  t_1^{\pm 1}, \dots, t_n^{\pm 1}
\right \rangle
}
\varinjlim_r
R_0 \left \langle  t_1^{\pm 1/p^r}, \dots, t_n^{\pm
1/p^r} \right \rangle $ with the evident $G$-action by roots of unity. 
Using \Cref{perfectoidtrunc}, one sees that the 
hypothesis (2) applies. 
Hypothesis (3) applies thanks to \Cref{ppowerrootex}. 
It follows that 
the comparison map 
\eqref{RSequiv} is an equivalence. 
In particular, the $K(1)$-local $\TR$ of $R$ is expressed as the inverse limit
of a diagram of the $K(1)$-local $\TR$ of various rings whose $p$-completion is
perfectoid.  
\end{example} 

\section{The analog of Thomason's spectral sequence}

\newcommand{\perfd}{\mathrm{Perfd}}
\newcommand{\pf}{\mathrm{perfd}}

In this section, we construct in certain cases an analog of Thomason's \'etale
descent 
spectral sequence for $L_{K(1)} K(-)$ in terms of \'etale cohomology,
cf.~\cite{Tho85, TT90}, for $L_{K(1)} \TR(-)$. Our construction splits into two
parts. First, we give a formula for $\TR$ and its $K(1)$-localization on the
 category of perfectoid rings. 
Second, we hypersheafify $L_{K(1)} \TR(-)$ on all rings in the $\arc$-topology and
compare 
$L_{K(1)} \TR(-)$ to this hypersheafification.

\subsection{$\arc$-cohomology}
In this section, we discuss the cohomology with respect to the $\arc$-topology,
mentioned briefly in the introduction (\Cref{arcpdefinitionintro}). 
This is a variant of the following topology,
cf.~\cite[Sec.~2.2.1]{CS}.\footnote{Throughout, for set-theoretic reasons one
should impose cardinality bounds, i.e., assume all the rings one allows into the
site to be of cardinality less than some uncountable strong limit
cardinal $\kappa$, so that the sites will be small.
However, the choice of $\kappa$ will not matter in all the constructions we
consider and we will consequently suppress it. } 

\newcommand{\arcph}{\mathrm{arc}_{\hat{p}}}
\begin{definition}[The {$p$-complete $\mathrm{arc}$-topology}] 
The \emph{$p$-complete $\mathrm{arc}$-topology} or  \emph{$\arcph$-topology}
is the finitary Grothendieck topology (cf.~\Cref{convsite}) on the opposite of the category of derived $p$-complete
commutative rings 
such that a map $R \to R'$ is a cover if for every $p$-complete rank $\leq
1$-valuation ring $V$ and map $R \to V$, there is an extension $V \to W$ of
$p$-complete rank $\leq 1$ valuation rings and a map $R' \to W$ fitting into a
commutative diagram,
\[ \xymatrix{
R \ar[d]  \ar[r] &  R' \ar[d]  \\
V \ar[r] &  W 
}.\]
We will also consider restrictions of the $\arcph$-topology to 
appropriate subcategories of the category of derived $p$-complete rings, such as the category 
$\perfd$ of perfectoid rings \cite[Sec.~3]{BMS1}.  
The \emph{$\arc$-topology} is defined similarly but we only consider rank $1$
$p$-complete valuation rings where $p \neq 0$. 
\end{definition} 

The $\arcph$-topology behaves well with respect to perfectoid rings; one knows
that 
the structure presheaf is a sheaf of rings with respect to this topology, and one even
has no higher cohomology \cite[Sec.~8]{Prisms}. 
The $\arc$-topology is a variant where in some sense we also impose derived
saturatedness conditions. 
Note that a functor is an $\arc$-sheaf if and only if it is an $\arcph$-sheaf
and annihilates any $\mathbb{F}_p$-algebra. 
Any derived $p$-complete ring can be covered in the $\arcph$-topology
(resp.~the $\arc$-topology) by a product of $p$-complete absolutely integrally
closed rank $1$ valuation rings (resp.~$p$-complete absolutely integrally
closed rank $1$ valuation rings where $p \neq 0$). 

We now give some examples of $\arc$-cohomology. To begin with, we consider the
simplest case of constant sheaves (or $p$-adically constant ones); the result is
that one essentially recovers the $p$-adic \'etale cohomology of the generic
fiber. 

\begin{construction}[Constant sheaves in the $\arc$-topology] 
Given an abelian group $M$, 
we can consider the associated constant sheaf of abelian groups in the
$\arc$-topology; its value on a derived $p$-complete ring $R$ is given by  $ H^0_{\et}( \spec(R[1/p]), M)$. 
To see this, we observe that the presheaf of abelian groups $R \mapsto 
 H^0_{\et}( \spec(R[1/p]), M)$ is a sheaf in the $\arc$-topology on derived
 $p$-complete rings 
\cite[Cor.~6.17]{BM18} (which assumes $M$ torsion; however, this is not
necessary since we are only working with $H^0$). 
It admits a map from the constant presheaf $M$, and the kernel and cokernel
vanish locally in the $\arc$-topology. 
\end{construction}

\begin{construction}[$p$-adically constant sheaves and Tate twists in the $\arc$-topology]
Consider the sheaf of rings $\mathbb{Z}_p$ in the $\arc$-topology, defined as the
$p$-completion of the constant sheaf associated to $\mathbb{Z}_p$, or
equivalently as the inverse limit of the constant sheaves 
$\mathbb{Z}/p^n \mathbb{Z}$; 
explicitly, $\mathbb{Z}_p(R) = H^0_{\arc}( \spec(R), \mathbb{Z}_p) = 
H^0_{\proet}( \spec(R[1/p]), \mathbb{Z}_p)$. 
We construct an invertible $\mathbb{Z}_p$-module $\mathbb{Z}_p(1)$ as the
$\arc$-sheafification of the presheaf 
$R \mapsto \mathbb{Z}_p(1)^{\mathrm{pre}}(R) := T_p(R^{\times})$.  
To check that this is an invertible module, we observe
(cf.~\cite[Prop.~3.30]{BM18}) that 
$\mathbb{Z}_p(1)^{\mathrm{pre}}(R)$ is an invertible
$\mathbb{Z}_p(R)$-module whenever $R$ is a product of absolutely
integrally closed valuation rings of mixed characteristic $(0, p)$.  
Note that $\mathbb{Z}_p(1)/p$ can equally be described as the $\arc$-sheaf 
associated to the presheaf $R \mapsto \mu_p(R) = R^{\times}[p]$, since $\arc$-locally all
elements have $p$th roots. 
\end{construction}

\begin{proposition}[$\arc$-cohomology as $p$-adic vanishing cycles] 
Let $R$ be any derived $p$-complete ring. Then there is a natural equivalence
\begin{equation} \label{arcasvanishingcycles} R \Gamma_{\arc}( \spec(R), \mathbb{Z}_p(i)) \simeq R \Gamma_{\proet}(
\spec(R[1/p]), \mathbb{Z}_p(i)), \end{equation}
where $\mathbb{Z}_p(i)$ on the right-hand-side refers to the usual Tate twist on
the pro-\'etale site of $\spec(R[1/p])$. 
\label{padicvanishingprop}
\end{proposition} 
\begin{proof} 
We know that the right-hand-side is a $D(\mathbb{Z}_p)^{\geq 0}$-valued sheaf for the $\arc$-topology by
\cite[Cor.~6.17]{BM18}, and
we obtain a map 
from the 
presheaf $ R \mapsto T_p( R^{\times})^{\otimes i}$
to the right-hand-side. 
Sheafifying in the $\arc$-topology and $p$-completing, we obtain a
map from the left-hand-side to the right-hand-side as in
\eqref{arcasvanishingcycles}. 
To see that \eqref{arcasvanishingcycles} is an equivalence of $\arc$-sheaves, 
it suffices to check on 
products of rank $1$ absolutely integrally closed, $p$-complete valuation rings of mixed
characteristic as these form a basis
(cf.~\cite[Prop.~3.29]{BM18}), which is handled by the next lemma.  
\end{proof}

\begin{lemma} 
Let $R$ be a product $\prod_{t \in T} V_t $ of  absolutely integrally
closed, $p$-complete valuation
rings. 
Then: 
\begin{enumerate}
\item $H^j_{\proet}( \spec(R[1/p]), \mathbb{Z}_p(i)) = 0$ for $ j > 0$. 
\item
The map 
$T_p(R^{\times}) \to 
H^0_{\proet}( \spec(R[1/p]), \mathbb{Z}_p(1)) $ is an isomorphism. 
\end{enumerate}
\label{padicnearbyproductval}
\end{lemma} 
\begin{proof} 
It suffices to prove both claims after reducing modulo $p$. 
Consider the functor 
$A \mapsto F_j(A) := H^j( \spec(\hat{A}_p[1/p]; \mathbb{F}_p(i))$. 
This functor commutes with finite products and filtered colimits by the
Gabber--Fujiwara theorem, cf.~\cite[Th.~6.11]{BM18} for an account. 
To prove 
that $F_j(R) = 0$ for $j > 0$
which implies 
(1), 
it suffices 
as in \cite[Cor.~3.17]{BM18}
to show that $F$
vanishes on every ultraproduct of the $\{V_t\}$ for each ultrafilter on $T$. 
But these ultraproducts are all absolutely integrally closed,
$p$-henselian valuation rings, whence (1). The claim (2) is proved similarly
using the map $A^{\times}[p]  \to 
H^0_{\proet}( \spec(A[1/p]), \mathbb{F}_p(1))$. 
\end{proof} 

Next we consider the cohomology of the structure presheaf.

\begin{construction}[$\arc$-cohomology of the structure presheaf] 
For a derived $p$-complete commutative ring $R$, we let $R \Gamma_{\arc}( \spec(R), \mathcal{O})$
denote the 
$\arc$-cohomology of the structure presheaf. 
\end{construction}

The $\arc$-topology restricts to a topology on the opposite of $\perfd$. 
Our first goal is to identify concretely $\arc$-cohomology on $\perfd$. 

\begin{construction}[Saturation] 
Let $R$ be a perfectoid ring. We have the natural quotient $R \to
R' := R/\mathrm{rad}((p))$, which is the universal map from $R$ to a perfect
$\mathbb{F}_p$-algebra. 
Note that $R' \otimes^L_R R' $ is discrete (and equivalent to $R'$) since
relative tensor products of perfectoid rings are $p$-completely discrete (and
perfectoid).

Let $J = \mathrm{rad}(p)$. 
We have $J \otimes^L_R J \xrightarrow{\sim} J$. 
It follows that one is in the setup of almost mathematics, 
and we have a derived almostification functor $(-)_* \colon D(R) \to D(R)$,
which is also given by $R \hom_R( J, -)$. 
We claim that in fact $J$ can be made explicit, and in particular that it has
projective dimension $\leq 1$, so that for any discrete $R$-module $M$, $M_* \in D(R)^{[0, 1]}$. 
Indeed, there exists an element $\omega \in R$ such that 
$\omega$ is a unit multiple of $p$
and such
that $\omega$ admits a compatible system of $p$-power roots $\{\omega^{1/p^n}\}_{n
\geq 0}$, cf.~\cite[Lem.~3.9]{BMS1}. We claim that as an $R$-module, there is an
equivalence
\begin{equation}  J \simeq \varinjlim \left( R \stackrel{\omega^{1-1/p}}{\to} R
\stackrel{\omega^{1/p -
1/p^2}}{\to} \dots \right).  \end{equation}
To see this, 
we observe first that the filtered colimit on the right-hand-side is a submodule
of 
$R$
given by the ideal $J' := \bigcup_{n} (\omega^{1/p^n})$ (i.e., given by multiplication
by $\omega^{1/p^n}$ on the $n$th term); we can see this by $\arcph$-descent to reduce to the case where 
$R$ is a product of rank $\leq 1$-valuation rings, in which case the claim is
clear. Now clearly $J' \subset J$ and 
$R/ J'$ is a ring of characteristic $p$ on which the Frobenius
is surjective.  To obtain $J' = J$, 
we use \cite[Lem.~3.10]{BMS1} to see that the Frobenius induces an isomorphism
$R/\omega^{1/p^n} \xrightarrow{\sim} R/\omega^{1/p^{n-1}}$ for $n \geq 1$. 
This implies that $R/J'$ is perfect, whence $J = J'$ as desired. 
\end{construction} 

\begin{definition}[Spherically complete fields] 
We recall (cf.~\cite[Ch.~4]{vanRooij78} for an account) that a nonarchimedean field is called \emph{spherically complete} if
every decreasing sequence of closed disks 
has nonempty intersection; in particular, any such field is complete. Any
nonarchimedean field admits an extension which is spherically complete and
algebraically closed.  
\end{definition}

\begin{lemma} 
\label{sphcompletelimvanishlem}
Let $C$ be spherically complete with ring of integers $\mathcal{O}_C$. 
Given any tower $\left\{M_i\right\}_{i \geq 1}$ of cyclic
$\mathcal{O}_C$-modules, we have $\varprojlim^1 M_i = 0$. 
\end{lemma} 
\begin{proof} 
Writing $\left\{M_i\right\}$ as the cokernel of an injective
map $\left\{N_i\right\}_{i \geq 1} \to \left\{N_i'\right\}_{i \geq 1}$ of
inverse systems with the
$ N_i'$ levelwise isomorphic to
$\mathcal{O}_{C}$ and using the long exact sequence, 
we reduce to the case that the $M_i$ are individually isomorphic to
$\mathcal{O}_C$ (in particular torsion-free). 
Moreover, 
upon passing to a cofinal subfamily, 
 we may assume the $\left\{M_i\right\}_{i \geq 1}$ form a descending
sequence of cyclic ideals $\left\{I_i\right\}_{i \geq 1} \subset \mathcal{O}_C$; using
the short exact sequence of inverse systems
$0 \to \left\{I_i\right\}_{i \geq 1} \to \left\{\mathcal{O}_C\right\}_{i \geq 1}
\to \left\{\mathcal{O}_C/I_i\right\}_{i \geq 1} \to 0$ (with the middle sequence
constant), we see that it suffices to show that 
$\mathcal{O}_C \to \varprojlim_i \mathcal{O}_C/I_i$ is surjective. But this is
precisely the definition of spherical completeness. 
\end{proof}

\begin{proposition}[$\arc$-cohomology of perfectoids] 
\label{arccohperfectoid}
The functor $R \mapsto R \Gamma_{\arc}( \spec(R), \mathcal{O})$ restricted on
$\perfd$ agrees with $R \mapsto R_*$. 
In particular, for $R \in \perfd$, 
$R \Gamma_{\arc}( \spec(R), \mathcal{O}) \in D(R)^{[0, 1]}$. 
\end{proposition} 
\begin{proof} 
Let $R$ be a perfectoid ring. 
First, we claim that $\arc$-cohomology of $R$ with
$\mathcal{O}$-coefficients can be calculated either on 
all derived $p$-complete rings 
or the subcategory of perfectoid rings (endowed with the $\arc$-cohomology).
To see this, we use that perfectoid rings form a basis for the $\arc$-topology. 
In fact, any ring admits an $\arc$-cover 
by a semiperfectoid ring $S$, and then one can take the perfectoidization 
\cite[Sec.~8]{Prisms} of $S$, which is an $\arc$-cover of $S$ (e.g., maps of
$S$ or its perfectoidization into
$p$-complete absolutely integrally closed valuation rings are the same). 
The rest of the claim is a general argument following from the fact that perfectoid rings form a
basis for the $\arc$-topology, cf.~\Cref{commutewithhypersheaf} below.

Now it suffices to prove that $R \mapsto R_*$ is a $D(\mathbb{Z})^{\geq
0}$-valued $\arc$-sheaf on $\perfd$ and
that the map $R \to R_*$ is locally an equivalence. The fact that it is an $\arc$-sheaf follows because
the structure presheaf is a $D(\mathbb{Z})^{ \geq 0}$-valued sheaf on the
$\arcph$-topology on $\perfd$ and 
$R \mapsto R_*$ annihilates perfect $\mathbb{F}_p$-algebras. The fact that $R
\to R_*$ is locally an equivalence follows by taking an $\arc$-cover by a product of
rings of integers in various spherically
complete, algebraically closed nonarchimedean fields of mixed characteristic;
for such rings, $R = R_*$. 
\end{proof}

Let $R_0$ be a perfectoid base ring, and let $R$ be a derived $p$-complete
$R_0$-algebra. 
One has the construction of the \emph{perfectoidization}
$R_{\pf} $ of $R$, a coconnective $E_\infty$-algebra under $R$ 
which has the property that if $R$ is semiperfectoid, then $R_{\pf}$ is discrete
and is the universal perfectoid ring that $R$ maps to. 
As shown in \cite[Sec.~8]{Prisms}, $R \mapsto R_{\pf}$ is the
$\arcph$-cohomology
of the structure presheaf on the category of derived $p$-complete $R_0$-algebras. 

\begin{proposition}[$\arc$-cohomology as saturated perfectoidization] 
For $R$ an $R_0$-algebra which is derived $p$-complete, we have a natural
equivalence 
of $E_\infty$-algebras
in $D(R_0)$,
$R \Gamma_{\arc}( \spec(R), \mathcal{O}) \simeq  (R_{\pf})_*$ (where almost
mathematics is taken relative to $R_0$ and the ideal $\mathrm{rad}(p)$). 
\end{proposition} 
\begin{proof} 
This is proved similarly as in 
\Cref{arccohperfectoid}. 
We have a natural (in $R$) map 
$R \to (R_{\pf})_*$. To show that it is the $\arc$-sheafification, it suffices
to show that 
the codomain is an $\arc$-sheaf, and that 
the map $R \to (R_{\pf})_*$ has cofiber which vanishes locally in the
$\arc$-topology; note also that we can compute the $\arc$-cohomology either
over all derived $p$-complete rings or over $R_0$-algebras (\Cref{overcat:ex}). 
The codomain is an $\arcph$-sheaf because of 
the identification of perfectoidization with $\arcph$-cohomology
\cite[Sec.~8]{Prisms}, and hence is an
$\arc$-sheaf since it annihilates $\mathbb{F}_p$-algebras. The 
cofiber of $R \to (R_{\pf})_*$ vanishes locally in the $\arc$-topology, as one
sees by  
working with perfectoid $R$ and using the argument of \Cref{arccohperfectoid}. 
\end{proof} 

\begin{example} 
\label{integralexample}
Suppose $R$ is  the $p$-completion of a ring
which is integral over the perfectoid ring $R_0$. 
Then the perfectoidization $R_{\pf}$ is discrete \cite[Th.~10.11]{Prisms}. 
It thus follows that 
$R\Gamma_{\arc}( \spec(R), \mathcal{O}) \in D(R)^{[0, 1]}$. 
\end{example}

\begin{construction}[Witt vector cohomology in the $\arc$-topology] 
We consider the presheaf $W(\mathcal{O})$ given by $R \mapsto W(R)$ 
and its $\arc$-cohomology (with Tate twists) $R \Gamma_{\arc}(\spec(R), W(\mathcal{O})(i))$. 
Since the Witt vector functor is 
endowed with Frobenius and Verschiebung operators, so is 
the construction
$R \mapsto R \Gamma_{\arc}(\spec(R), W(\mathcal{O})(i))$. 
\end{construction} 

In our setting, we can think of the Witt vector cohomology considered above as a one-parameter
(along $V$) deformation
of the structure presheaf cohomology, especially in light of the next result.

\begin{proposition} 
\label{Vcompleteness}
For any ring $R$, $R \Gamma_{\arc}( \spec(R), W(\mathcal{O})(i))$ is $p$-complete
and complete with respect to the Verschiebung. 
\end{proposition} 
\begin{proof}

By base-changing to $\mathbb{Z}_p[\zeta_{p^\infty}]$, we may assume without loss 
of generality that $i = 0$.  
Then this follows from \Cref{commutehypproducts} below, since the presheaf 
$W(\mathcal{O})$ commutes with finite products is complete with respect to $(p, V)$. 
\end{proof}

\begin{corollary} 
Let $R$ be finite and finitely presented over a perfectoid ring. 
Then $R \Gamma_{\arc}( \spec(R), W(\mathcal{O})(i)) \in D(\mathbb{Z}_p)^{[0,
1]}$ for any $i$. 
\label{cohamplitudeWOi}
\end{corollary} 
\begin{proof} 
By $p$-completeness and $V$-completeness 
(\Cref{Vcompleteness}), 
it suffices to prove the analogous statement for 
$R \Gamma_{\arc}( \spec(R), \mathcal{O}(i)/p) $. 
Replacing $R$ with $R[\zeta_p]$ and using $(\mathbb{Z}/p)^{\times}$-Galois
descent along this extension, we may assume that $ R$ contains a primitive $p$th
root of unity. 
This lets us reduce to the case $i = 0$, whence the result follows 
from \Cref{integralexample}. 
\end{proof}

We note finally that one can recover the $p$-adic nearby cycles 
(cf.~\Cref{padicvanishingprop}) as the fixed points of Frobenius 
on Witt vector $\arc$-cohomology. 
This will not be used in the sequel. 

\begin{proposition} 
\label{arcgivesnearbycycles}
For any ring $R$, there is a natural equivalence
for each $i$,
\begin{equation} \label{nearbyisFrobfixed} R \Gamma_{\arc}( \spec(R), \mathbb{Z}_p(i)) \simeq 
R \Gamma_{\arc}( \spec(R), W(\mathcal{O})(i))^{F = 1}.
\end{equation}
\end{proposition} 
\begin{proof} 
We have a natural map $\mathbb{Z}_p \to W(\mathcal{O})^{F=1}$ of presheaves.
Twisting and sheafifying in the $\arc$-topology, we obtain the map from left to
right in 
\eqref{nearbyisFrobfixed}. 
To see that it is an equivalence, it suffices to 
check that it is an equivalence on products (over some indexing set $T$) of 
rings of the form $\mathcal{O}_C$, for $C$ spherically complete and
algebraically closed of mixed characteristic $(0, p)$; now we can trivialize the
Tate twists, so can assume $i = 0$. 
Using \Cref{padicnearbyproductval}, we find that the left-hand-side is discrete
and simply given by $\prod_T \mathbb{Z}_p$. 
The right-hand-side is also given by $\prod_T \mathbb{Z}_p $ by 
\Cref{Frobfixedex} below. 
\end{proof}

\begin{lemma} 
\label{Frobfixedex}
Let $C$ be an algebraically closed complete nonarchimedean field. 
Then the natural map $\mathbb{Z}_p \to W(\mathcal{O}_C)^{F=1}$ is an
equivalence. 
\end{lemma} 
\begin{proof} 
It suffices to work modulo $p$, i.e., to show that the natural map 
$\mathbb{F}_p \to (W(\mathcal{O}_C)/p)^{F = 1}$ is an equivalence. But the Witt
vector Frobenius reduces modulo $p$ to the ordinary Frobenius on the ring
$W(\mathcal{O}_C)/p$. By the Artin--Schreier  sequence, 
$(W(\mathcal{O}_C)/p)^{F = 1}$
is the mod $p$ \'etale cohomology of $W(\mathcal{O}_C)/p$. 
This is unchanged by taking the quotient by the ideal $V W(\mathcal{O}_C)/p$,
which squares to zero; therefore, it is also the \'etale cohomology of
$\mathcal{O}_C/p$, or equivalently of the residue field $k$; this in turn is
$\mathbb{F}_p$ as $C$ is algebraically closed and hence so is $k$. 
\end{proof}

\subsection{$\THH^{C_{p^r}}$ of perfectoid rings}

Here we calculate the fixed points of $\THH$ of perfectoid rings, generalizing the
main result of \cite{Hes06}; these results are known to experts. 
Recall that for any $p$-complete ring $R$, we have $\pi_0 \TR(R; \mathbb{Z}_p)
\simeq W(R)$, cf.~\cite[Th.~F]{HM97}. 

Our strategy is to treat the case of a $p$-torsionfree perfectoid ring by a direct
spectral sequence argument and appeal to known results for a perfect
$\mathbb{F}_p$-algebra; we will glue both cases together using an excision
argument. 
To begin with, we consider the characteristic $p$ case. 
For perfect fields, this result appears in  \cite[Sec.~5]{HM97}. 
\begin{lemma} 
\label{TRofperfectFP}
Let $R$ be a perfect $\mathbb{F}_p$-algebra. 
Then: 
\begin{enumerate}
\item  
$\pi_*(\THH(R; \mathbb{Z}_p)^{C_{p^r}}) \simeq W_{r+1}(R)[\sigma_r]$ for $|\sigma_r| =
2$. 
\item
$\TR(R; \mathbb{Z}_p) \simeq H W(R)$. 
\end{enumerate}
\end{lemma} 
\begin{proof} 
The Segal conjecture holds for $R$: in fact, $\varphi \colon \THH(R;
\mathbb{Z}_p) \to
\THH(R ; \mathbb{Z}_p)^{tC_p}$ is $(-3)$-truncated, cf.~\cite[Sec.~6]{BMS2}. 
Therefore, we have that 
$\THH(R; \mathbb{Z}_p)^{C_{p^r}} \to \THH(R; \mathbb{Z}_p)^{hC_{p^r}}
\xrightarrow{\varphi^{hC_{p^r}}}\THH(R; \mathbb{Z}_p)^{tC_{p^{r+1}}}$ are
equivalences on connective covers, cf.~\cite[Cor.~II.4.9]{NS18}. 
Using the homotopy fixed point spectral sequence (or \cite[Sec.~6]{BMS2}), we
find 
$\pi_*( \THH(R; \mathbb{Z}_p)^{tS^1}) \simeq W(R)[ \sigma^{\pm 1}]$. 
Now $\THH(R; \mathbb{Z}_p)^{tC_{p^{r+1}}} \simeq \THH(R;
\mathbb{Z}_p)^{tS^1}/p^{r+1}$ by \cite[Lem.~IV.4.12]{NS18}. 
Combining these facts, we see that  the first claim follows as $W_{r+1}(R) = W(R)/p^{r+1}$. 

For the second claim, 
note that $\THH(R; \mathbb{Z}_p)$ is a module in $\cycsp$ over $K(\mathbb{F}_p; \mathbb{Z}_p) \simeq
H\mathbb{Z}_p$. 
By \Cref{K1locfromcycl}, we have that 
$\TR(R; \mathbb{Z}_p) \to L_{K(1)} \TR(R; \mathbb{Z}_p)$ is 
$0$-truncated; since the target vanishes, it follows that $\TR(R; \mathbb{Z}_p)$
is $0$-truncated, whence the result as we know $\pi_0\TR(R; \mathbb{Z}_p)$. 
\end{proof} 

In the following, we will say that an $E_\infty$-ring $A$ is \emph{weakly even
periodic} if the odd homotopy groups of $A$ vanish, and if for all $m, n \in
\mathbb{Z}$, the map $\pi_{2m}(A) \otimes_{\pi_0(A)} \pi_{2n}(A) \to
\pi_{2(n+m)}(A)$ is an isomorphism; in particular, $\pi_2(A)$ is an invertible
$\pi_0(A)$-module, and Zariski locally $\pi_*(A)$ is a Laurent polynomial
algebra over $\pi_0(A)$ on a degree two class.

\begin{lemma}  \label{pullbackeven}
Let $A, B, C$ be weakly even periodic $E_\infty$-rings and fix maps $A\to C, B
\to C$ such that $\pi_0(A) \to \pi_0(C)$ is surjective. Then $A \times_C B$ is
weakly even periodic. 
\end{lemma} 
\begin{proof} 
It follows that $ \pi_*(A) \otimes_{\pi_0(A)} \pi_0(C) \xrightarrow{\sim}
\pi_*(C)$ and similarly 
$ \pi_*(A) \otimes_{\pi_0(A)} \pi_0(B) \xrightarrow{\sim}
\pi_*(B)$. 
Now 
since $\pi_0(A \times_C B), \pi_0(A), \pi_0(B), \pi_0(C)$ form a Milnor square, 
the category of finitely generated 
projective 
$\pi_0(A \times_C B)$-modules is the homotopy pullback of the categories of
finitely generated projective $\pi_0(A)$-modules, $\pi_0(B)$-modules, and
$\pi_0(C)$-modules, cf.~\cite[Sec.~2]{Mi71}. 
The result now follows. 
\end{proof}

\begin{construction}[$p$-torsionfree quotients of perfectoid rings] 
Let $R$ be a perfectoid ring, and let $I \subset R$ be the ideal of $p$-power torsion. 
Then $p I = 0$, cf.~\cite[Prop.~4.19]{BMS2}.

We have in fact that $R/I$ is perfectoid as well. 
To see this, we use the theory of perfectoidizations 
\cite[Sec.~7--8]{Prisms}. 
Let $R' = (R/p)_{\mathrm{perf}}$. 
For every $p$-complete valuation ring $V$ with a map $R \to V$ which does not
annihilate $I \subset R$, we find that $V$ has characteristic $p$ and we obtain
a unique extension $R' \to V$. 
Therefore, applying 
\cite[Cor.~8.11]{Prisms}, we obtain a pullback square of perfectoid rings,
\begin{equation} \label{ptorsionfreesq} \xymatrix{ 
R \ar[d]  \ar[r] &  R' \ar[d]  \\
(R/I)_{\pf}   \ar[r] & (R'/I)_{\pf}
}. \end{equation}
Since the vertical arrows are surjective (cf.~\cite[Th.~7.4]{Prisms}), this is a Milnor square of rings. 
Since the terms on the right side are both $\mathbb{F}_p$-algebras, it follows
that the kernel of $R \to (R/I)_{\pf}$ is annihilated by $p$; since it contains
$I$, it must be equal to $I$ and we have $(R/I)_{\pf} = R/I$. 

In particular, this shows that any perfectoid ring naturally fits into a
Milnor square 
involving a $p$-torsionfree perfectoid ring and a map of perfect $\mathbb{F}_p$-algebras. 
The diagram \eqref{ptorsionfreesq} is in addition a 
homotopy pushout square of $E_\infty$-rings, i.e., there are no higher
$\mathrm{Tor}$ terms, since everything involved is
perfectoid. It follows by \cite{LT19} that \eqref{ptorsionfreesq} induces a
pullback after applying any localizing invariant.  
\end{construction}

\begin{proposition}
\label{TRnperfectoid}
Let $R$ be a perfectoid ring. Then for each $r \geq 1$, we have: 
\begin{enumerate}
\item $\pi_{*} \THH(R; \mathbb{Z}_p)^{C_{p^r}}$ is concentrated in even degrees.  
\item $\pi_{2} \THH(R; \mathbb{Z}_p)^{C_{p^r}} $ is an invertible  module over 
$\pi_0 \THH (R; \mathbb{Z}_p)^{C_{p^r}} \simeq W_{r+1}(R)$. 
\item  The multiplication map 
$\mathrm{Sym}^t 
(\pi_{2} \THH(R; \mathbb{Z}_p)^{C_{p^r}}) \to \pi_{2t}\THH(R;
\mathbb{Z}_p)^{C_{p^r}}$ is an
isomorphism for all $t \geq 0$. 
\end{enumerate}
\end{proposition} 
\begin{proof} 
Recall \cite[Sec.~6]{BMS2} that we have a noncanonical isomorphism
$\pi_* \THH(R; \mathbb{Z}_p)^{tC_p} \simeq R[\sigma^{\pm 1}]$ and that 
$\varphi \colon \THH(R; \mathbb{Z}_p) \to \THH(R; \mathbb{Z}_p)^{tC_p}$ exhibits
the source as the connective cover of the target. 
Therefore, we have that 
$\THH(R; \mathbb{Z}_p)^{C_{p^r}} \to \THH(R; \mathbb{Z}_p)^{hC_{p^r}}
\xrightarrow{\varphi^{hC_{p^r}}}\THH(R; \mathbb{Z}_p)^{tC_{p^{r+1}}}$ are
equivalences on connective covers, cf.~\cite[Cor.~II.4.9]{NS18}.

Our next claim is that 
$ \THH(R; \mathbb{Z}_p)^{tC_{p^{r+1}}}$ is a weakly even periodic 
$E_\infty$-ring for any $r \geq 0$. 
In case $R$ is $p$-torsionfree, this follows from the (degenerate) Tate spectral
sequence applied to $\THH(R; \mathbb{Z}_p)$, whose homotopy groups form a
polynomial algebra over $R$ on a class in degree $2$. 
In case $R$ is an $\mathbb{F}_p$-algebra, the claim also follows from
\Cref{TRofperfectFP}. 
We now treat the case of a general perfectoid $R$. 
In view of the Milnor square
\eqref{ptorsionfreesq} which is also a homotopy pushout of $E_\infty$-rings and
the main result of \cite{LT19}, we find that $\THH(-;
\mathbb{Z}_p)^{tC_{p^{r+1}}}$ carries
\eqref{ptorsionfreesq} to a pullback of $E_\infty$-ring spectra.  
Moreover, on $\pi_0(-)$ this square 
yields a Milnor square since $\pi_0 \THH(-; \mathbb{Z}_p)^{tC_{p^{r+1}}} \simeq
W_{r+1}(-)$ for perfectoids via the previous paragraph and \cite[Th.~F]{HM97}. 
Via \Cref{pullbackeven}, the claim of weak even 
periodicity follows. 

Now the above paragraphs show that 
$\THH(R; \mathbb{Z}_p)^{C_{p^r}}$ is the connective cover of a weakly even
periodic $E_\infty$-ring, and we already know its $\pi_0$ 
is given by $W_{r+1}(R)$, whence the result. 
\end{proof}

\subsection{Description of $\TR$ of perfectoids}
\renewcommand{\inf}{\mathrm{inf}}

Let $C$ be spherically complete and algebraically closed of mixed
characteristic $(0, p)$. Then we recall some of
the additional rings attached to $C$. 
We have the tilt ${C}^{\flat}$, the ring of integers
$\mathcal{O}_{C^{\flat}} \subset C^{\flat}$, the 
Fontaine ring $A_{\inf} = A_{\inf}(\mathcal{O}_C)  = W(\mathcal{O}_{C^{\flat}})$. 
Choosing a compatible system $(1, \zeta_p, \zeta_{p^2}, \dots )$ of primitive $p^n$th roots
in $C$, we obtain an element $\epsilon \in \mathcal{O}_{C^{\flat}}$ 
and let $[\epsilon]$ denote the corresponding element of $A_{\inf}$. 
We have the map $\theta \colon A_{\inf} \to \mathcal{O}_C$, which exhibits the
source as the universal $p$-adically pro-nilpotent thickening of the latter; in
particular, this leads to a map $A_{\inf} \to W(\mathcal{O}_C)$, which is a
surjection in this case \cite[Lem.~3.23]{BMS1}. 

\begin{lemma}[$\TR$ in the spherically complete case] 
\label{TRsphcompl}
Let $C$ be a spherically complete, algebraically closed nonarchimedean field 
of mixed
characteristic $(0, p)$
and let $\mathcal{O}_C
\subset C$ be the ring of integers. 
Let $\beta \in \pi_2 \TR(\mathcal{O}_C; \mathbb{Z}_p)$ be the Bott element
(arising from the image of the cyclotomic trace). 
Then $\TR_*(\mathcal{O}_C; \mathbb{Z}_p) \simeq
W(\mathcal{O}_C)[\beta]$. 
\end{lemma} 
\begin{proof} 
Since everything is $p$-complete, 
it suffices to see that $\TR(\mathcal{O}_C; \mathbb{F}_p)/\beta$ is $0$-truncated. 
Now for each $n$, each homotopy group in $\pi_*(\THH^{C_{p^n}}(\mathcal{O}_C; \mathbb{F}_p))$
is a cyclic module over 
$W(\mathcal{O}_C)/p$ (\Cref{TRnperfectoid}), which in turn is a quotient of $A_{\inf}(\mathcal{O}_C)/p =
\mathcal{O}_{C^{\flat}}$. Moreover, the homotopy groups are concentrated in even
degrees.

Now $C^{\flat}$ is
spherically complete, cf.~the proof of \cite[Lem.~3.23]{BMS1}. 
Using the previous paragraph, 
\Cref{sphcompletelimvanishlem},
and the Milnor exact sequence, it follows that 
$\pi_*( \TR( \mathcal{O}_C; \mathbb{F}_p))$ is 
concentrated in even degrees. 
By 
\Cref{K1locallemma} and 
\Cref{perfectoidtrunc}, we also find that 
the cofiber of $\beta$ is $2$-truncated, so 
$\pi_*( \TR( \mathcal{O}_C; \mathbb{F}_p)/\beta)$ is concentrated in degrees $0$
and $2$, and is given by $W(\mathcal{O}_C)/p$ in degree zero. 
It thus suffices to show that 
the zeroth Postnikov section
$\TR(\mathcal{O}_C; \mathbb{Z}_p)/\beta \to H W(\mathcal{O}_C)$
induces an isomorphism on $\pi_0 (-)^{tS^1}$. 
But we have by 
\cite[Cor.~10]{AN}, 
$\TR(\mathcal{O}_C; \mathbb{Z}_p)^{tS^1} = \THH(\mathcal{O}_C;
\mathbb{Z}_p)^{tS^1}$, and on homotopy groups this is 
$A_{\inf}[\sigma^{\pm 1}]$ (cf.~\cite[Sec.~6]{BMS2}). 
Moreover, 
with respect to the above identifications, we have
that $\beta$ is a $\mathbb{Z}_p^{\times}$-multiple of 
$( [\epsilon] -1 )\sigma$, cf.~\cite[Th.~1.3.6]{HN}. 
Therefore, we have
$ \pi_0 ( \TR(\mathcal{O}_C; \mathbb{Z}_p)/\beta )^{tS^1} \simeq A_{\inf}/(
[\epsilon] - 1) [\sigma^{\pm
1}]$, whence the result since 
$A_{\inf}/( [\epsilon] -1 ) \xrightarrow{\sim} W(\mathcal{O}_C)$ by
\cite[Lem.~3.23]{BMS1}. 
\end{proof}

\begin{proposition}[$\TR$ of perfectoid rings] 
\label{TRofperfectoid}
For $R \in \perfd$ and for $i > 0$, we have natural equivalences
\begin{equation} \tau_{[2i-1, 2i]} \TR(R; \mathbb{Z}_p) \simeq 
R \Gamma_{\arc}( \spec(R), W(\mathcal{O})(i))[2i].
\end{equation}
For $R \in \perfd$, we have 
for $i \in \mathbb{Z}$,
\begin{equation} \label{TRpieces}   \tau_{[2i-1, 2i]} \left( L_{K(1)}\TR(R; \mathbb{Z}_p) \right) \simeq 
R \Gamma_{\arc}( \spec(R), W(\mathcal{O})(i))[2i] . \end{equation}
\end{proposition} 
\begin{proof} 
The structure presheaf on $\perfd$ defines an $\arcph$-sheaf,
\cite[Prop.~8.9]{Prisms}.  
Taking the inverse limit in \Cref{TRnperfectoid} over the restriction maps (and accounting for possible $\lim^1$ terms),
we find in particular that 
$R \mapsto \tau_{[2i-1, 2i]} \TR(R; \mathbb{Z}_p)$ defines an
$\arcph$-sheaf of spectra on $\perfd$. 
Since for $i > 0$, this functor annihilates perfect $\mathbb{F}_p$-algebras, 
it in fact defines an $\arc$-sheaf on $\perfd$.

The cyclotomic trace gives a map 
of $\mathbb{Z}_p$-modules
$T_p( R^{\times}) \to \pi_2 \TR(R; \mathbb{Z}_p)$. 
Since $\pi_0( \TR(R; \mathbb{Z}_p) ) \simeq W(R)$, we obtain natural maps 
$f_i \colon W(R) (i)\to  \pi_{2i}( \TR(R; \mathbb{Z}_p))$ for $i > 0$. To complete the proof, it
suffices to show that: 

\begin{enumerate}[a)]
\item  
$f_i$ is an isomorphism for a cofinal (in the $\arc$-topology) collection of $R \in \perfd$. 
\item
$\pi_{2i-1}( \TR(R; \mathbb{Z}_p)) =0 $ for a cofinal (in the
$\arc$-topology) collection of $R \in
\perfd$. 
\end{enumerate}

We will take this collection to be the set of products of perfectoid rings of
the form $\mathcal{O}_C$, for $C$ a spherically complete and algebraically
closed  nonarchimedean field of mixed characteristic $(0, p)$. 
First, \Cref{TRnperfectoid} shows that the construction $R \mapsto \THH(R;
\mathbb{Z}_p)$ commutes with products for $R \in \perfd$; inductively and taking
the limit we find that $R \mapsto \TR(R; \mathbb{Z}_p)$ commutes with products
for $R \in \perfd$. 
Thus, it suffices to verify the above claims for $R = \mathcal{O}_C$ itself,
and that follows from \Cref{TRsphcompl}.

For the claim about $L_{K(1)} \TR(R; \mathbb{Z}_p)$, note first that 
the first part of the proof and \Cref{perfectoidtrunc} 
(and \eqref{ptorsionfreesq} to reduce to the $p$-torsionfree case)
shows easily that $R \mapsto L_{K(1)} \TR(R; \mathbb{Z}_p)$ defines an
$\arc$-sheaf on $\perfd$ (cf.~also the remarks at the beginning of the proof
of \Cref{perfectoidtrunc}, concerning when $L_{K(1)}(-)$ commutes with homotopy limits).
It then follows from the above (by inverting the Bott element over perfectoid rings
containing $p$-power roots of unity) that the homotopy groups of the $\arc$-sheaf
$L_{K(1)} \TR(R; \mathbb{Z}_p)$ on $\perfd$ are given by 
$\pi_{2i} \simeq W(\mathcal{O})(i)$, and we obtain a filtration with associated
gradeds the right-hand-side of 
\eqref{TRpieces}
via the Postnikov filtration as $\arc$-sheaves. 
Note that the Postnikov filtration is always exhaustive, and it is complete
because it is complete on the $\arc$-sheaf $R \mapsto \tau_{\geq 1} L_{K(1)} \TR(R) \simeq
\tau_{\geq 1} \TR(R; \mathbb{Z}_p)$. 
Since the relevant associated graded pieces are in homological degrees $[2i-1,
2i]$ by \Cref{cohamplitudeWOi}, the result follows. 
\end{proof} 

\subsection{The main results}

In this subsection, we prove \Cref{sseqthmintro} from the introduction. 
Our strategy is to compare $L_{K(1)} \TR(-)$ with its
$\arc$-hypersheafication. 
\newcommand{\shv}{\mathrm{Shv}}

\begin{construction}[The invariant $\left( L_{K(1)} \TR(-) \right)^{\sharp} $]
We define the functor $(L_{K(1)} \TR(-))^{\sharp}$ on derived $p$-complete rings
as the $\arc$-hypersheafification of the functor 
$L_{K(1)} \TR(-)$ on derived $p$-complete rings. 
We have a comparison map 
\begin{equation} L_{K(1)} \TR(-) \to 
\left( L_{K(1)} \TR(-) \right)^{\sharp}. \label{sharpcompmap} \end{equation}
For a ring which is not necessarily derived $p$-complete, we define 
$\left( L_{K(1)} \TR(-) \right)^{\sharp}$ as 
that of its derived $p$-completion. 
\end{construction}

To analyze this construction, we will use the results about 
$L_{K(1)} \TR$ of perfectoids proved in the previous section, as well as some
general tools from the appendix. 
On $\perfd$, \Cref{TRofperfectoid} implies that $L_{K(1)} \TR(-)$ defines an
$\arc$-hypersheaf. 
Using \Cref{commutewithhypersheaf}, it follows that the
restriction of $\left(L_{K(1)}
\TR(-)\right)^{\sharp}$ to $\perfd$ is just $L_{K(1)} \TR(-)$ again, i.e.,
\eqref{sharpcompmap} is an equivalence.\footnote{Alternatively, one could build 
$\left( L_{K(1)} \TR(-) \right)^{\sharp}$ using the unfolding construction of
\Cref{arcunfolding}, by starting with the $\arc$-hypersheaf $L_{K(1)} \TR(-)$ on
$\perfd$.
} 

\begin{proposition} 
For any ring $R$, 
there is a natural, complete exhaustive $\mathbb{Z}_{}$-indexed filtration on 
$\left( L_{K(1)} \TR(R) \right)^{\sharp}$, 
denoted 
$\fil^{\geq \ast}\left( L_{K(1)} \TR(R) \right)^{\sharp}$, with associated
graded terms
$\gr^i \left( L_{K(1)} \TR(R) \right)^{\sharp} \simeq R \Gamma_{\arc}(
\spec(R), W(\mathcal{O})(i))[2i]$.  
\end{proposition} 
\begin{proof} 
The filtration in question is the $\arc$-Postnikov filtration (on the derived
$p$-completion of $R$). 
Note that Postnikov towers converge for hypercomplete $\arc$-sheaves
(\Cref{commutehypproducts}). 
For the identification of the graded pieces (or equivalently, the sheafified
homotopy groups), it suffices by descent (cf.~\Cref{arcunfolding}) to work with the subcategory $\perfd$,
where the result follows from \eqref{TRpieces}. 
\end{proof}

We do not know in general for which $R$ the comparison map \eqref{sharpcompmap} is an
equivalence; for such $R$, one obtains a ``motivic'' filtration on $L_{K(1)}
\TR(R)$ with associated graded terms the $\arc$-cohomology 
complexes 
$R \Gamma_{\arc}(
\spec(R), W(\mathcal{O})(i))[2i]$.
Here we will show that the comparison map is an equivalence in certain formally
smooth cases, using the pro-Galois descent result (\Cref{progaldesc}). 
\begin{theorem} 
Let $R_0$ be a $p$-torsionfree perfectoid ring. 
Let $R$ be a formally smooth $p$-complete $R_0$-algebra. 
Then $L_{K(1)} \TR(R) \xrightarrow{\sim} (L_{K(1)} \TR(R))^{\sharp}$. 
\label{sharpcompACcase}
\end{theorem} 
\begin{proof} 
First, we reduce to the case where $R_0$ admits a compatible system of $p$-power
roots of unity. By Andr\'e's lemma \cite[Th.~7.12]{Prisms}, we know that 
there exists a $p$-completely faithfully flat perfectoid $R_0$-algebra $R_0'$ which has
this property (e.g., is absolutely integrally closed). 
Now we have by descent
\begin{equation} \label{TRdesc}  \TR(R; \mathbb{Z}_p) \simeq \mathrm{Tot}( \TR(R \otimes_{R_0} R_0' ;
\mathbb{Z}_p) \rightrightarrows \TR( R \otimes_{R_0}R_0' \otimes_{R_0} R_0';
\mathbb{Z}_p) \triplearrows \dots ).  \end{equation}
Since in high degrees all of the terms in the above totalization agree with their
$K(1)$-localization by \Cref{perfectoidtrunc}, it follows that  
the 
descent property \eqref{TRdesc}
holds for $L_{K(1) } \TR(-)$ as well 
(cf.~the beginning of the proof of \emph{loc.~cit.}). 
Moreover, $(L_{K(1)} \TR(-))^{\sharp}$ satisfies descent for the map $R \to R
\otimes_{R_0} R_0'$ by construction.
Therefore, we reduce to the case where $R_0$ contains
$\widehat{\mathbb{Z}_p[\zeta_{p^\infty}]}$.

Working locally on $\mathrm{Spf}(R)$, we can assume that $R$ 
receives a map from 
$R_0\left \langle t_1^{\pm 1}, \dots, t_n^{\pm 1}\right\rangle$ 
which is \'etale mod $p$. 
We consider the 
extension $R_\infty = R \otimes_{\mathbb{Z}[\zeta_{p^\infty}, t_1^{\pm 1},
\dots, t_n^{\pm 1}]} 
\mathbb{Z}[\zeta_{p^\infty}, t_1^{\pm 1/p^\infty},
\dots, t_n^{\pm 1/p^\infty}]$ and the evident $\mathbb{Z}_p(1)^n$-action on $R_\infty$. 
As in \Cref{smoothoverOC}, we find from 
\Cref{progaldesc} that the natural map induces an equivalence
\[ L_{K(1)} \TR(R) \xrightarrow{\sim} 
\mathrm{Tot}\left( L_{K(1)} \TR(R_\infty) \rightrightarrows L_{K(1)} \TR(
\mathrm{Fun}_{\mathrm{cts}}( \mathbb{Z}_p(1)^n, R_\infty)) \triplearrows \dots
\right).
\]
Now this is also true for $(L_{K(1)}\TR(-))^{\sharp}$, because the above
augmented cosimplicial ring is an $\arc$-hypercover (strictly speaking, for that we replace
all rings by their derived $p$-completions). 
But now this gives $L_{K(1)} \TR(R) \xrightarrow{\sim} 
(L_{K(1)}\TR(R))^{\sharp}$, since they agree on perfectoids and every term in
the above cosimplicial resolution has perfectoid $p$-completion. 
\end{proof}

\begin{theorem} 
Suppose $R$ is a formally smooth $\mathcal{O}_K$-algebra, where $K$ is a complete
discretely valued field
of mixed 
characteristic $(0, p)$
whose residue field $k$ satisfies $[k:k^p]< \infty$. Then 
$L_{K(1)} \TR(R) \xrightarrow{\sim} 
\left( L_{K(1)} \TR(R) \right)^{\sharp}$. 
\end{theorem} 
\begin{proof} 
We 
let $G = \mathrm{Gal}(\overline{K}/K)$ and consider the $G$-action on 
$S = R \otimes_{\mathcal{O}_K} \mathcal{O}_{\overline{K}}$. 
As in \Cref{DVRex}, our assumptions imply that $G$ has finite cohomological dimension and moreover
that the map 
\eqref{RSequiv} is an equivalence. 
Therefore, it suffices to show that 
$L_{K(1)} \TR( A) \xrightarrow{\sim} (L_{K(1)} \TR( A))^{\sharp} $
when $A$ is one of 
$S, \mathrm{Fun}_{\mathrm{cts}}(G, S), \dots$. 
But these are all (up to $p$-completions) formally
smooth over perfectoids,  
whence the claim by 
\Cref{sharpcompACcase} and descent. 
\end{proof}

\begin{remark} 
Suppose $L_{K(1)} \TR(R) \xrightarrow{\sim} \left(L_{K(1)} \TR(R)
\right)^{\sharp}$. 
On Frobenius fixed points, one recovers the Thomason \cite{Tho85,
TT90} filtration (i.e., the pro-\'etale Postnikov filtration) on 
$L_{K(1)} \TC(R) \simeq L_{K(1)} K(R[1/p])$ (cf.~\cite{BCM} for this
identification), whose associated gradeds are given
by 
$\gr^i \simeq R \Gamma_{\proet}( \spec(R[1/p]), \mathbb{Z}_p(i))[2i]$, cf.~
\Cref{arcgivesnearbycycles}. 
\end{remark}

\begin{remark} 
The above filtration on $L_{K(1)} \TR(-)$ 
and the calculations of $\TR$ of smooth algebras over a DVR in 
\cite{HM03, HM04, GH06}
suggest that the cohomology 
$R \Gamma_{\arc}( \spec(R), W(\mathcal{O})(i))$ should be related to 
the absolute de Rham--Witt 
complex. 
The work \cite{Mor18} also suggests that for smooth algebras over a perfectoid
base containing all $p$-power roots of unity, 
$R \Gamma_{\arc}( \spec(R), W(\mathcal{O}))$ should be related to the relative de
Rham--Witt complex (over the perfectoid base). 
\end{remark} 

\appendix

\section{Topological preliminaries}

In this appendix, we record some basic topological preliminaries about
(hyper)sheaves of spectra.

\begin{remark}[Conventions for sites] 
We will for simplicity work only with 
sites of the following nice form (cf.~\cite[Sec.~A.3.2]{SAG}). 
Let $\mathcal{C}$ be a category with pullbacks and finite coproducts, 
such that coproducts distribute over pullbacks and are disjoint. Suppose $\mathcal{C}$ is
equipped with a class of morphisms $S = S_{\mathcal{C}}$ which contains all
isomorphisms and is
stable under composition and pullback. 
We equip $\mathcal{C}$ with the Grothendieck topology where 
a collection $\left\{X_i \to X\right\}_{i \in I}$ is a covering if  there exists
a finite subset $I' \subset I$ such that $\bigsqcup_{i \in I'} X_i \to X$
can be refined by a map belonging to $S$. 
In this case, a presheaf on $\mathcal{C}$ with values in an $\infty$-category
$\mathcal{D}$ with all small limits is a sheaf if and only if it 
carries finite coproducts in $\mathcal{C}$ to finite products in $\mathcal{D}$
and if it satisfies \v{C}ech descent for maps in $S$, cf.~\cite[Sec.~A.3.3]{SAG}. 
\label{convsite}
\end{remark} 

\begin{example} 
\begin{enumerate}
\item The small or big \'etale site of a qcqs scheme (where we only allow
qcqs schemes) is an example, with $S$ the class of \'etale surjections. 
\item The $\arc$-topology or $\arcph$-topology on the opposite of the category of derived
$p$-complete rings. 
\item The $\arc$-topology or $\arcph$-topology on the opposite of the category of perfectoid
rings. 
\end{enumerate}
\end{example} 

\begin{remark}[Examples of continuous functors] 
Let $\mathcal{C}, \mathcal{C}'$ be sites as in \Cref{convsite}. 
Let $u \colon \mathcal{C}' \to \mathcal{C}$ be a 
functor which preserves finite coproducts and pullbacks, as well as morphisms
in the respective classes $S_{\mathcal{C}}, S_{\mathcal{C}'}$. 
It follows that if $\sF$ is a sheaf (with values in any $\infty$-category
$\mathcal{D}$ with all small limits) on $\mathcal{C}$, then $\sF \circ u$ is a
sheaf on $\mathcal{C}'$. 
These are examples of continuous functors; the notion can be defined for more
general sites, cf.~\cite[Exp.~III.1]{SGA4}. 
\label{exofcontfunctor}
\end{remark}

\begin{construction}[Sheaves of spectra] 
Given a site $\mathcal{C}$ (as in 
\Cref{convsite}), we let $\psh(\mathcal{C}, \sp)$ 
denote the $\infty$-category of presheaves of spectra on $\mathcal{C}$ and 
$\shv(\mathcal{C}, \sp) \subset \psh(\mathcal{C}, \sp)$ denote the 
subcategory of sheaves of spectra \cite[Sec.~1.3]{SAG}. 
We equip these both 
with their canonical $t$-structures, cf.~\cite[Sec.~1.3.2]{SAG}. 

Given $u \colon \mathcal{C}' \to \mathcal{C}$ as in \Cref{exofcontfunctor}, 
we obtain a right adjoint and left $t$-exact functor
$(-) \circ u \colon \shv(\mathcal{C}, \sp) \to \shv(\mathcal{C}', \sp)$. 
It has a left adjoint $u_! \colon \shv(\mathcal{C}', \sp) \to \shv(\mathcal{C},
\sp)$ 
given by left Kan extension along $u$ followed by sheafification. 
By adjunction, necessarily $u_!$ is right $t$-exact.  
\end{construction}

\newcommand{\nl}{\mathrm{null}}
\newcommand{\shvh}{\mathrm{Shv}_{\mathrm{hyp}}}

\begin{construction}[Hypersheaves of spectra] 
Let $\mathcal{C}$ be a site as in \Cref{convsite}. 
Any presheaf $\sF$ of spectra on $\mathcal{C}$ fits into a unique
fiber sequence
of spectra
\begin{equation}  \sF_{\nl} \to \sF \to \sF^{\sharp}
\label{hypcofib} \end{equation}
where $\sF_{\nl}$ has trivial sheafified homotopy groups 
and $\sF^{\sharp}$ is a
hypercomplete sheaf of spectra, i.e., for every presheaf $\mathcal{G}$ with
trivial sheafified homotopy groups, we have $\hom(\mathcal{G}, \sF^{\sharp}) =0
$. 
If $\sF$ is a sheaf of spectra, then the cofiber sequence shows that $\sF_{\nl}$
is also a sheaf of spectra; it is then $\infty$-connective with respect to the
$t$-structure on $\shv(\mathcal{C}, \sp)$. 
We let $\shvh( \mathcal{C}, \sp) \subset \shv(\mathcal{C}, \sp)$ denote 
the full subcategory of hypercomplete sheaves. 
We refer to \cite[Sec.~2]{CM19} for an exposition of some of these
constructions, which go back to \cite{Ja87, DHI04}. 
\end{construction} 

\begin{proposition} 
\label{commutewithhypersheaf}
Let $u \colon \mathcal{C}' \to \mathcal{C}$ be a morphism of sites as in \Cref{convsite}
preserving finite coproducts and  pullbacks and carrying the class of arrows
$S_{\mathcal{C}'}$ into $S_{\mathcal{C}}$. 
Suppose that 
for any object $X' \in \mathcal{C}'$ and a finite covering family $\{ Y_i \to
u(X')\}_{i \in I}$ in $\mathcal{C}$, 
there is a refinement which is the image under $u$ of a finite covering family 
$\left\{Y_i' \to X'\right\}_{i \in I}$ in $\mathcal{C}'$. 

Then the restriction functor 
$(-) \circ u \colon \psh(\mathcal{C}, \sp) \to \psh(\mathcal{C}', \sp)$ commutes with
hypersheafification (and in particular preserves hypercomplete sheaves). 
\end{proposition}

\begin{proof} 
Using the cofiber sequence \eqref{hypcofib}, we see that it suffices to prove
that $(-) \circ u \colon  \shv( \mathcal{C}, \sp) \to \shv(\mathcal{C}', \sp)$
preserves both the subclasses of objects with trivial homotopy groups and
hypercomplete objects.

Our hypotheses imply that 
if a presheaf of abelian groups on $\mathcal{C}$ has trivial sheafification,
then 
its pullback to $\mathcal{C}'$ has trivial sheafification. 
Therefore, $(-) \circ u $ preserves objects with trivial homotopy groups. 
 It thus suffices to show that if $\sF \in \shvh(\mathcal{C}, \sp)$,
then the sheaf $\sF \circ u \in \shv(\mathcal{C}', \sp)$ is also hypercomplete;
this will not use the assumption in the second sentence of the statement.
Equivalently,  given $\mathcal{G} \in \shv(\mathcal{C}', 
\sp)$ which is $\infty$-connective, it suffices to show that $\hom_{\shv(\mathcal{C}', \sp)}(\mathcal{G}, \sF
\circ u) =0 $.  
By adjointness, this is $\hom_{\shv(\mathcal{C}, \sp)} ( u_! \mathcal{G},
\sF)$; since $u_! \colon \shv(\mathcal{C}', \sp) \to \shv(\mathcal{C}, \sp)$ is right $t$-exact and therefore preserves $\infty$-connective objects and $\sF$ is
hypercomplete, this is contractible. 
\end{proof} 

\begin{example}[Overcategories] 
\label{overcat:ex}
Suppose $a \in \mathcal{C}$ and $\mathcal{C}' = \mathcal{C}_{/a}$ is the
overcategory of $a$, with the induced topology. Then the natural forgetful functor $\mathcal{C}' \to
\mathcal{C}$ clearly satisfies the conditions of 
\Cref{commutewithhypersheaf}. In particular, the hypersheafification of a
presheaf on $\mathcal{C}$ when restricted to $\mathcal{C}_{/a}$ is the
hypersheafification
of the restriction to $\mathcal{C}_{/a}$. 
\end{example}

A direct consequence is that given an appropriate site with a ``basis,'' 
hypersheaves of spectra can be entirely recovered from their values on the
basis. 
The result is a direct analog of \cite[Th.~4.1, Exp.~III]{SGA4}, and appears
(for sheaves of spaces) in
\cite[App.~A]{Aoki}. 

\begin{proposition} 
Let $\mathcal{C}$ be a site as in \Cref{convsite}. 
Let $\mathcal{C}' \subset \mathcal{C}$ be a full subcategory closed under finite
coproducts and fiber products, and define $S_{\mathcal{C}'}$ to be the
intersection of $S_{\mathcal{C}}$ with $\mathcal{C}'$. 
Suppose every object $X \in \mathcal{C}$ admits a map $Y \to X$ in
$S_{\mathcal{C}}$ with $Y \in \mathcal{C}'$. 
Then the restriction functor 
$\shv(\mathcal{C}, \sp) \to \shv(\mathcal{C}', \sp)$ restricts to an
equivalence on hypercomplete objects. 
\end{proposition} 
\newcommand{\sG}{\mathcal{G}}
\begin{proof} 
Let $u \colon \mathcal{C}' \subset \mathcal{C}$ be the inclusion, so we have a
restriction functor $(-) \circ u \colon \psh(\mathcal{C}, \sp) \to
\psh(\mathcal{C}', \sp)$. 
By \Cref{commutewithhypersheaf}, it restricts to a functor on hypercomplete
sheaves, so we have $(-) \circ u \colon \shvh(\mathcal{C}, \sp) \to
\shvh(\mathcal{C}', \sp)$. 
This last functor has a left adjoint $L \colon 
\shvh(\mathcal{C}', \sp) \to \shvh(\mathcal{C}, \sp)
$, given by $\sF \mapsto L\mathcal{F} := (\mathrm{Lan}_u
\mathcal{F})^{\sharp}$, i.e., $L$ is obtained by applying the left Kan extension $\mathrm{Lan}_u
\colon \psh(\mathcal{C}', \sp) \to \psh(\mathcal{C}, \sp)$ followed by
hypersheafification $(-)^{\sharp}$. 
We now show that $L$ is fully faithful. Indeed, we have
for $\sF, \sG \in \shvh(\mathcal{C}', \sp)$,
\begin{equation} 
\hom_{\shvh(\mathcal{C}, \sp) }( L \mathcal{F}, L \mathcal{G})  = 
\hom_{\shvh( \mathcal{C}', \sp)}(\mathcal{F}, (\mathrm{Lan}_u \mathcal{G})^{\sharp} \circ u
).
\label{exprleftadj}
\end{equation} 
Now since $u$ is fully faithful,  since hypersheafification commutes with $(-)
\circ u$ by \Cref{commutewithhypersheaf}, and since 
$\mathcal{G}$ is
already hypercomplete, we have
$(\mathrm{Lan}_u \mathcal{G})^{\sharp} \circ u = \mathcal{G} $. 
Therefore, the right-hand-side of \eqref{exprleftadj} simplifies to $\hom_{\shvh(\mathcal{C}',
\sp)}(\mathcal{F}, \mathcal{G})$ as desired.

Our hypotheses imply that the restriction 
functor
$(-) \circ u$ 
is conservative on hypercomplete sheaves, because it is conservative on sheaves
of abelian groups. 
Since the restriction functor is conservative and has a fully faithful 
left adjoint, the result follows. 
\end{proof} 

\begin{example}[Unfolding in the $\arc$-topology]  
\label{arcunfolding}
The 
natural forgetful functor establishes an equivalence of $\infty$-categories
between $\arc$-hypersheaves of spectra on all derived $p$-complete rings and
$\arc$-hypersheaves of
spectra on perfectoid rings. 
\end{example}

Finally, we include a basic observation about commuting hypersheafification and
certain products when the site $\mathcal{C}$ is sufficiently large (e.g., the
$\arc$-site). 
This fact is closely related to the theory of replete topoi,
cf.~\cite[Sec.~3]{BSproet}. Compare \cite[Prop.~3.3.3]{BSproet} for the second
part of the next result for the derived category of abelian sheaves, or
equivalently hypercomplete $H\mathbb{Z}$-module sheaves of spectra.  
Note in particular that it applies to the $\arc$-site. 
This follows because 
$\mathrm{arc}$-covers in the sense of \cite{BM18} are closed under filtered
colimits of rings 
\cite[Cor.~2.20]{BM18} and because a map of derived $p$-complete rings $R \to
R'$ is an $\arc$-cover if and only if $R \to R' \times R/p \times R[1/p]$ is an
$\mathrm{arc}$-cover. 

\newcommand{\pshc}{\psh_{\sqcup}}

In the following, we write $\pshc(\mathcal{C}) \subset \psh(\mathcal{C})$ for
the subcategory of presheaves which carry finite coproducts
to finite products. 
\begin{proposition} 
\label{commutehypproducts}
Let $\mathcal{C}$ be a site as in 
\Cref{convsite}. 
Suppose $\mathcal{C}$ admits countable filtered limits. 
Suppose moreover that 
if $\dots \to X_i \to X_{i-1} \to \dots \to X_0$ is a 
sequence of arrows in $S$, then $\varprojlim_i X_i \to X_0$ belongs to $S$. 
Then:  
\begin{enumerate}
\item  
The hypersheafification functor, 
$(-)^{\sharp} \colon \pshc(\mathcal{C}) \to \shvh(\mathcal{C})$ commutes 
with countable products and limits along $\mathbb{Z}_{\geq 0}$-indexed towers. 
\item
For any $\sG \in \shvh(\mathcal{C}) \subset \shv(\mathcal{C})$, the Postnikov tower
of $\sG$ (as a sheaf of spectra) converges. 
\end{enumerate}
\end{proposition} 
\begin{proof} 
For (1), 
since 
hypersheaves are always closed (inside presheaves) under arbitrary limits and
since $\mathbb{Z}_{\geq 0}$-indexed limits can be built from countable products, it
suffices to show (via the unique cofiber sequence \eqref{hypcofib}) that if 
$\left\{\sF_i\right\}_{i \in \mathbb{N}}$ is a countable family of presheaves in
$\pshc(\mathcal{C})$
with trivial hypersheafification (that is, trivial sheafified homotopy groups),
then the product presheaf $\prod_{i \in \mathbb{N}} \sF_i$ has trivial
hypersheafification.  
But this is just a claim about the presheaves of abelian groups
$\left\{\pi_j( \sF_i)\right\}_{i \in \mathbb{N}}$ for each $j$. 
Explicitly, given a class $\alpha = (\alpha_i) \in \prod_{i \in \mathbb{N}} \pi_j(
\sF_i(X))$ for some $X \in \mathcal{C}$, 
we find for each $i$ a cover $X_i \to X$
in $S$ which annihilates\footnote{Here we use that we are working with presheaves which
carry finite coproducts to finite products, so that we can reduce to working
with covers consisting of one morphism.} $\alpha_i \in \pi_j( \sF_i(X))$ and then form the cover 
$X_1 \times_X X_2 \times_X  \dots$ of $X$, which annihilates $\alpha$.

Now let $\mathcal{G} $ be a hypercomplete sheaf of spectra on $\mathcal{C}$. 
The Postnikov tower of $\mathcal{G} \in \shv(\mathcal{C})$ is obtained by taking the presheaf
truncations $\tau^{\mathrm{pre}}_{\leq n} \mathcal{G}$ and applying the hypersheafification (or
sheafification, since these objects are truncated), i.e., one forms 
$\{ (\tau^{\mathrm{pre}}_{\leq n} \mathcal{G})^{\sharp} \}$, which is a tower in
$\pshc(\mathcal{C})$. 
Since Postnikov towers converge for presheaves of spectra, i.e., $\mathcal{G}
\simeq \varprojlim_n \tau^{\mathrm{pre}}_{\leq n} \mathcal{G}$, and 
we have just seen that $(-)^{\sharp} \colon \pshc(\mathcal{C}) \to
\shvh(\mathcal{C})$ 
commutes with limits along $\mathbb{Z}_{\geq 0}$-indexed towers, we find $\mathcal{G} \simeq \varprojlim_n
(\tau_{\leq n}^{\mathrm{pre}} \mathcal{G})^{\sharp}$ as desired. 
\end{proof}

\bibliographystyle{amsalpha}
\bibliography{galoisdesc}

\end{document}